\documentclass[12pt]{amsart}
\usepackage{amsfonts}
\usepackage{amssymb,amstext,amscd,amsmath,mathdots}
\usepackage{fullpage}
\usepackage[all]{xy}
\usepackage[usenames,dvipsnames]{color}
\allowdisplaybreaks[4] 
\usepackage{graphicx}
\usepackage{bbding}
\usepackage[misc]{ifsym}
\usepackage{cite}

\numberwithin{equation}{section}
\theoremstyle{plain}
\newtheorem{thm}{Theorem}[section]
\newtheorem{cor}[thm]{Corollary}
\newtheorem{prop}[thm]{Proposition}
\newtheorem{lem}[thm]{Lemma}

\theoremstyle{definition}

\newtheorem{rems}[thm]{Remarks}
\newtheorem{defn}[thm]{Definition}

\newtheorem{ques}{Question}

\newcommand{\bC}{{\mathbb{C}}}

\newcommand{\bN}{{\mathbb{N}}}

  \newcommand{\B}{{\mathcal{B}}}

\renewcommand{\H}{{\mathcal{H}}}

\renewcommand{\P}{{\mathcal{P}}}

\renewcommand{\phi}{\varphi}
\newcommand{\upchi}{{\raise.35ex\hbox{$\chi$}}}

\newcommand{\ran}{\operatorname{Ran}}

\begin{document}

\title{On the Power Set of Quasinilpotent Operators in Banach Spaces}

\author[C. L. Hu]{ChaoLong Hu}
\address{ChaoLong HU: School of Mathematics\\Jilin University\\Changchun 130012\\P.R. CHINA;
School of Mathematics and Science\\Henan Institute of Science and Technology\\ Xinxiang 453003\\ P.R. CHINA}
\email{huzl21@mails.jlu.edu.cn}

\author[Y. Q. Ji]{YouQing Ji}
\address{YouQing JI: School of Mathematics\\Jilin University\\Changchun 130012\\P.R. CHINA}
\email{jiyq@jlu.edu.cn}

\author[D. H. Liang]{DingHao Liang}
\address{DingHao LIANG: School of Mathematics\\Jilin University\\Changchun 130012\\P.R. CHINA}\email{liangdh24@mails.jlu.edu.cn}

\thanks{Supported by NNSF of China Grant No.12271202.}
\begin{abstract}
  For a quasinilpotent operator $T$ on a Banach space $X$,
  Douglas and Yang defined
  $k_{x}=\limsup\limits_{\lambda\rightarrow 0}\frac{\ln\|(\lambda-T)^{-1}x\|}{\ln\|(\lambda-T)^{-1}\|}$
  for each non-zero vector $x$, and called $\Lambda(T)=\{k_x:x\neq 0\}$ the
  {\it power set} of $T$.
  In this paper, we prove that $\Lambda(T)$ always contains $1$ for every quasinilpotent operator $T$ on $X$.
  Moreover, we introduce the concept of a Banach space $X$ having uniform multiplicity infinity and prove that some classical Banach spaces possess this property.
  As an application,
  we show that if $\sigma\subset [0,1]$ is right closed and contains $1$, then there exists a quasinilpotent operator $T$ on a class of Banach spaces with uniform multiplicity infinity such that $\Lambda(T)=\sigma$.
\end{abstract}
\maketitle


\section{Introduction}

Throughout this paper, we denote by $X$ a complex Banach space, and by $\H$ a complex Hilbert space.
The symbol $\B(X)$ stands for the set of bounded linear operators on $X$.
For $T\in \mathcal{B}(X)$, let $\sigma(T)$ be the spectrum of $T$.
A closed subspace $M\subset X$ is called an invariant subspace under the operator $T$ if $TM\subset M$.
The famous \emph{invariant subspace problem} asks:

{\em Does every bounded linear operator on a separable, infinite dimensional, complex Hilbert space have a non-trivial closed invariant subspace?}

The invariant subspace problem is still unsolved.
Nevertheless,
numerous results indicate that spectral theory and the invariant subspace problem are closely related.
For instance, it is well known that the Riesz decomposition theorem implies that if the spectrum of a bounded linear operator $T$ on $X$ is not connected, then $T$ has a non-trivial invariant subspace.
Brown, Chevreau and Pearcy \cite{BCP1979} proved that if $T\in\B(X)$ with $\|T\|=1$ and $\sup\{|f(z)|:z\in\sigma(T)\cap \mathbb{D}\}=\|f\|_{\infty}$ for every $f$ in $H^{\infty}(\mathbb{D})$,
then $T$ has a non-trivial invariant subspace.
Brown \cite{Brown1987} showed that a hyponormal operator $T$ has a non-trivial invariant subspace whenever $R(\sigma(T))\neq C(\sigma(T))$, where $C(\sigma(T))$ denotes the Banach algebra of all continuous functions defined on the spectrum of $T$ and $R(\sigma(T))$ is the closure in $C(\sigma(T))$ of the algebra of rational functions with poles off $\sigma(T)$.
For more introductions on the invariant subspace problem, one can see \cite{CP2011,RR03,Sar74}.

A special case for which the invariant subspace problem is
still open is that of quasinilpotent operators on $\H$.
Recall that an operator $T$ is said to be {\em quasinilpotent} if $\sigma(T)=\{0\}$.
Since the spectrum of a quasinilpotent operator is a singleton, no further information can be obtained from it.
Therefore, the invariant subspace problem for quasinilpotent operators is particularly difficult.
We refer the readers to \cite{FJKP04,FJKP05,FP05,Read97,Tcaciuc2019} and the references therein for more information along this line.

The \emph{power set} $\Lambda(T)$ of a quasinilpotent operator $T$ on $X$ was introduced by  Douglas and Yang (see \cite{DY2,DY3}).
For the reader's convenience, we recall it below.
\begin{defn}
Suppose that $T\in\mathcal{B}(X)$ is quasinilpotent and $x\in X\setminus\{0\}$. Let
$$k_{T}(x)=\limsup\limits_{\lambda\rightarrow 0}\frac{\ln\|(\lambda-T)^{-1}x\|}{\ln\|(\lambda-T)^{-1}\|}.$$
Usually, we briefly write $k_x$ instead of $k_{T}(x)$ if there is no confusion.
Set $\Lambda(T)=\{k_x : x\ne 0\}$, and call it the {\em power set} of $T$.
\end{defn}
\begin{rems}
For a quasinilpotent operator $T\in\B(X)$, we list some basic facts about $\Lambda(T)$.
\end{rems}
\begin{itemize}
  \item [(1)] From
  $$\frac{\|x\|}{|\lambda|+\|T\|}\leq\|(\lambda-T)^{-1}x\|\leq\|(\lambda-T)^{-1}\|\|x\|,~~ \forall \lambda\ne 0,$$
  it follows that $\Lambda(T)\subset [0,1]$.
  \item [(2)]
  Since $\lim\limits_{\lambda\rightarrow 0}\|(\lambda-T)^{-1}\|=+\infty$ when $\sigma(T)=\{0\}$,
  it holds that $k_{(\alpha x)}=k_x$ for all $\alpha\ne 0$ and $x\ne 0$.
  Thus, $\Lambda(T)=\{k_x:\|x\|=1\}$.
  \item [(3)]
  It is not hard to check that the power set is invariant under dilation, i.e.,
  $\Lambda(\alpha T)=\Lambda(T)$ for all $\alpha\neq 0$.
  \item [(4)]
  $\Lambda(T)=\Lambda(A)$ if $A$ is similar to $T$ {\cite[Proposition 5.2]{DY3}}.
\end{itemize}

Recall that a closed subspace $M\subset X$ is called {\em hyperinvariant subspace} for $T$ if $AM\subset M$ for every $A$ commuting with $T$.
The following result of
Douglas and Yang established a link between the power set and hyperinvariant subspace.
\begin{thm}[see {\cite[Proposition 7.1, Corollary 7.2]{DY3}}]\label{pthy}
Let $T\in \mathcal{B}(X)$ be a quasinilpotent operator.
For $\tau\in[0,1]$, write $M_\tau=\{x : k_{x}\leq\tau\}\cup\{0\}$.
Then, $M_\tau$ is a linear subspace of $X$, and $A(M_\tau)\subset M_\tau$ for every $A$ commuting with $T$.

In particular, when $\Lambda(T)$ contains two different points $\tau$ with closed $M_\tau$,
$T$ has a nontrivial hyperinvariant subspace.
\end{thm}
There has been a series of works exploring
the quasinilpotent operator by the power set.
The interested readers can refer
\cite{HeZhu,HuJi,JiLiu,LY2018} for more details.
Comparing with the facts that $\sigma(A)$ is compact for every $A$ in $\B(\H)$ and each nonempty compact subset of $\mathbb{C}$ is the spectrum of some operators in $\B(\H)$.
Ji and Zhang \cite{JiZhang} proved that the power set $\Lambda(T)$ of a quasinilpotent operator $T$ on $\H$ must be right closed (that is, $\sup E \in \Lambda(T)$ for every nonempty subset $E$ of $\Lambda(T)$).
Furthermore,
they also showed that for each right closed subset $\sigma$ of $[0,1]$ containing $1$, there exists a quasinilpotent operator $T\in\B(\H)$ such that $\Lambda(T)=\sigma$.
However, this result requires the mild assumption that $1\in\sigma$.
So, one may naturally consider the following question.

\begin{ques}\label{Q1}
Whether or not $1$ is always in the power set of every quasinilpotent operator?
\end{ques}

Recall that an infinite-dimensional Banach space $X$ is said to be \emph{hereditarily indecomposable} (briefly HI) if every closed
infinite-dimensional subspace of $X$ cannot be written as the direct sum of two infinite-dimensional closed subspaces.
Such spaces were first exhibited by Gowers and Maurey in \cite{GM1993}.
One of the main properties of an HI space $X$ is that it has few operators,
namely, every $T\in\B(X)$ is a multiple of the identity plus a strictly singular operator, where an operator $T\in\B(X)$ is said to be strictly singular if there is no closed infinite-dimensional subspace $M\subset X$ such that $T|_{M}$ is one-to-one with closed range.
In this case, the spectrum of any operator $T\in\B(X)$ is countable with at most one limit point; see for instance Corollary 6.35 in \cite{BM2009}.
Therefore, it is not difficult to see that there are compact subsets in $\bC$ which are not the spectrum of any operator in HI space.
In contrast to this fact about the spectrum, the following question is meaningful.
\begin{ques}\label{Q2}
For each right closed subset $\sigma$ of $[0,1]$ containing $1$, is it true that $\Lambda(T)=\sigma$ for some quasinilpotent operator $T\in\B(X)$?
\end{ques}

In this paper, we will continue to investigate some fundamental properties of the
power set.
Our first main result is
\begin{thm}\label{dlA}
Let $T\in\B(X)$ be a quasinilpotent operator, then $1\in\Lambda(T)$.
\end{thm}
Theorem \ref{dlA} provides a positive answer to Question \ref{Q1}.
Furthermore, to investigate Question \ref{Q2}, we introduce the following concept.
\begin{defn}\label{I:uniform}
Let $X$ be an infinite-dimensional Banach space.
We say that $X$ is of {\em uniform multiplicity infinity} if there exists a system of matrix units $\{u_{j,k}\}_{j,k=0}^{\infty}$ in $\B(X)$, where $\{u_{j,k}\}_{j,k=0}^{\infty}$ satisfied that
\begin{enumerate}
  \item [(a)]
  $\sup\{\|u_{j,k}\|:j,k\geq 0\}<\infty$;
  \item [(b)]
  $u_{j,k}u_{l,m}=\delta_{kl} u_{j,m}$ for any $j,k,l,m$, where
  $$\delta_{kl}=\begin{cases}0,&
  k\neq l;\\
  1, & k=l;\end{cases}$$
  \item [(c)]
  $x=\sum_{k=0}^{\infty}u_{k,k}x$ for every $x\in X$,
  with convergence in norm.
\end{enumerate}
\end{defn}
We give an affirmative answer to Question \ref{Q2} on a class of infinite-dimensional Banach spaces with uniform multiplicity infinity and establish the following result.

\begin{thm}\label{dlB}
Suppose that $X$ is separable and of uniform multiplicity infinity, and $\ran u_{0,0}$ has an unconditional basis and is infinite-dimensional.
If $\sigma\subset [0,1]$ is right closed and contains $1$, then there exists a quasinilpotent operator $T\in\B(X)$ such that $\Lambda(T)=\sigma$.
\end{thm}

The paper is organized as follows.
In Section 2, we prove that the power set $\Lambda(T)$ always contains $1$ for each quasinilpotent operator $T\in\B(X)$.
In Section 3, we study Banach spaces with uniform multiplicity infinity and show that some classical Banach spaces have this property.
In Section 4, we prove that if $\sigma\subset [0,1]$ is right closed and contains $1$, then there exists a quasinilpotent operator $T$ on a class of Banach spaces with uniform multiplicity infinity such that $\Lambda(T)=\sigma$.

\section{Proof of Theorem \ref{dlA}}
The main goal of this section is to present the proof of Theorem \ref{dlA}.
\begin{lem}[see \cite{HuJi2025}]\label{RClosed}
Suppose that $T\in\mathcal{B}(X)$ is quasinilpotent, then $\Lambda(T)$ is right closed.
\end{lem}

\begin{proof}[Proof of Theorem~\ref{dlA}]
If $T$ is nilpotent with $T^n=0$ but $T^{n-1}\neq0$, then $\Lambda(T)=\{\frac{j}{n}:j=1,2,\cdots,n\}$ (see \cite[Example 6.1]{DY3}).
So, $1\in \Lambda(T)$.
Now, let $T\in\B(X)$ be a quasinilpotent but not nilpotent operator.
Assume that $1\notin \Lambda(T)$.
By Lemma \ref{RClosed}, we have $\Lambda(T)$ is right closed.
So, there exists a positive number $c$ with $0<c<1$ such that $\Lambda(T)\subset [0,c]$.
Let $q=\frac{c+1}{2}<1$, then $k_x<q$ for any $x\in X$.
Next, we will construct a non-zero vector $y\in X$ such that $k_y=1$ to derive a contradiction.

\begin{itemize}
  \item [\textbf{STEP 1.}] Finding the non-zero vector $y\in X$ by induction.
\end{itemize}
\begin{enumerate}
  \item [(1)] Induction part $1$.

  Suppose that $e_1\in X$ and $\|e_1\|=1$, then $k_{e_1}<q$.
  So there exists $r_1$ in $(0,1)$ such that
  $$\|(\lambda-T)^{-1}e_1\|<\|(\lambda-T)^{-1}\|^q, \ \  \|(\lambda-T)^{-1}\|>2^2, \ \ \forall |\lambda|\leq r_1.$$
  Set $s_1=\frac{1}{4}$ and $x_1=s_1e_1$.
  Then
  $$\|(\lambda-T)^{-1}x_1\|=s_1\|(\lambda-T)^{-1}e_1\|<\|(\lambda-T)^{-1}\|^q,\ \ \forall |\lambda|\leq r_1.$$
  For each $\lambda\neq 0$, since
  $$\|(\lambda-T)^{-1}\|=\sup\{\|(\lambda-T)^{-1}x\|: x\in X,\|x\|=1\},$$
  we can pick $e=e(\lambda)\in X$ with $\|e\|=1$ so that
  $$\|(\lambda-T)^{-1}e\|\geq (1-\frac{1}{2!4^2})\|(\lambda-T)^{-1}\|.$$
  So one can find $\lambda_2$ with $0<|\lambda_2|<r_1$ such that
  $$\|(\lambda_2-T)^{-1}\|=\max\{\|(\lambda-T)^{-1}\|:|\lambda|=|\lambda_2|\}$$
  and pick $e_2\in X$ with $\|e_2\|=1$ such that
  $$\|(\lambda_2-T)^{-1}e_2\|\geq(1-\frac{1}{2!4^2})\|(\lambda_2-T)^{-1}\|.$$

  \item [(2)] Induction part $2$.

  Let $s_2=\frac{1}{4^2}$ and $x_2=s_2e_2$.
  Set $y_2=x_1+x_2$, then $y_2\neq 0$ and $k_{y_2}<q$.
  So there exists $r_2\in(0,1)$ and $0<r_2<|\lambda_2|$ such that
  $$\|(\lambda-T)^{-1}y_2\|<\|(\lambda-T)^{-1}\|^q,\ \  \|(\lambda-T)^{-1}\|>3^3,\ \
  \forall |\lambda|\leq r_2.$$
  For each $\lambda\neq 0$, there exists $e=e(\lambda)\in X$ and $\|e\|=1$ such that
  $$\|(\lambda-T)^{-1}e\|\geq (1-\frac{1}{3!4^3})\|(\lambda-T)^{-1}\|.$$
  So, one can find $\lambda_3$ with $0<|\lambda_3|<r_2$ such that
  $$\|(\lambda_3-T)^{-1}\|=\max\{\|(\lambda-T)^{-1}\|:|\lambda|=|\lambda_3|\}$$
  and pick $e_3\in X$ with $\|e_3\|=1$ such that
  $$\|(\lambda_3-T)^{-1}e_3\|\geq(1-\frac{1}{3!4^3})\|(\lambda_3-T)^{-1}\|.$$

  \item [(3)] Induction part $3$.

  Let $s_3=\frac{1}{4^3}$ and $x_3=s_3e_3$.
  Set $y_3=\sum_{i=1}^{3}x_i$, then $y_3\neq 0$ and $k_{y_3}<q$.
  So there exist $r_3\in(0,1)$ and $0<r_3<|\lambda_3|$ such that
  $$\|(\lambda-T)^{-1}y_3\|<\|(\lambda-T)^{-1}\|^q, \ \
  \|(\lambda-T)^{-1}\|>4^4, \ \
  \forall |\lambda|\leq r_3.$$
  For each $\lambda\neq 0$, there exists $e=e(\lambda)\in X$ and $\|e\|=1$ such that
  $$\|(\lambda-T)^{-1}e\|\geq (1-\frac{1}{4!4^4})\|(\lambda-T)^{-1}\|.$$
  So, one can find $\lambda_4$ with $0<|\lambda_4|<r_3$ such that
  $$\|(\lambda_4-T)^{-1}\|=\max\{\|(\lambda-T)^{-1}\|:|\lambda|=|\lambda_4|\}$$
  and pick $e_{4}\in X$ with $\|e_4\|=1$ such that
  $$\|(\lambda_4-T)^{-1}e_4\|\geq(1-\frac{1}{4!4^4})\|(\lambda_4-T)^{-1}\|.$$

  \item [(4)]
  Suppose that we have completed the Induction part $j$, $1 \leq j \leq k$. We will next get the Induction part $(k+1)$.

  Let $s_k=\frac{1}{4^k}$ and $x_k=s_ke_k$.
  Set $y_k=\sum_{i=1}^{k}x_i$, then $k_{y_k}\neq 0$ and $k_{y_k}<q$.
  So there exist $r_k\in(0,1)$ and $0<r_k<|\lambda_k|$ such that
  $$\|(\lambda-T)^{-1}y_k\|<\|(\lambda-T)^{-1}\|^q,\ \ \|(\lambda-T)^{-1}\|>(k+1)^{k+1},
  \ \ \forall |\lambda|\leq r_{k}.$$
  For each $\lambda\neq 0$, there exists $e=e(\lambda)\in X$ and $\|e\|=1$ such that
  $$\|(\lambda-T)^{-1}e\|\geq (1-\frac{1}{(k+1)!4^{k+1}})\|(\lambda-T)^{-1}\|,\ \ \forall \lambda\neq 0.$$
  So, one can find $\lambda_{k+1}$ with $0<|\lambda_{k+1}|<r_k$ such that
  $$\|(\lambda_{k+1}-T)^{-1}\|=\max\{\|(\lambda-T)^{-1}\|:|\lambda|=|\lambda_{k+1}|\}$$
  and pick $e_{k+1}\in X$ with $\|e_{k+1}\|=1$ such that
  $$\|(\lambda_{k+1}-T)^{-1}e_{k+1}\|\geq(1-\frac{1}{(k+1)!4^{k+1}})\|(\lambda_{k+1}-T)^{-1}\|.$$

  \item [(5)]
  By induction, we can find $\{r_k\}_{k=1}^{\infty}$, $\{\lambda_k\}_{k=2}^{\infty}$, $\{e_k\}_{k=1}^{\infty}$ ,$\{x_k\}_{k=1}^{\infty}$ and $\{y_k\}_{k=1}^{\infty}$ such that
\begin{itemize}
  \item [(a)]
  $1>r_1>|\lambda_2|>r_2>|\lambda_3|>r_3>\cdots>r_k>|\lambda_{k+1}|>\cdots>0$;
  \item [(b)]
  $\|(\lambda_k-T)^{-1}\|=\max\{\|(\lambda-T)^{-1}\|:|\lambda|=\lambda_k\}$;
  \item [(c)]
  $\|(\lambda_k-T)^{-1}e_k\|\geq(1-\frac{1}{k!k^k})\|(\lambda_k-T)^{-1}\|$;
  \item [(d)]
  $s_k=\frac{1}{4^k}$, $x_k=s_k e_k$ and $y_k=\sum_{i=1}^{k}x_i$ for
  $k\geq 1$;
  \item [(f)]
  $\|(\lambda-T)^{-1}y_k\|<\|(\lambda-T)^{-1}\|^q$ and  $\|(\lambda-T)^{-1}\|>(k+1)^{k+1}$ for any $|\lambda|\leq r_k$.
\end{itemize}
\end{enumerate}

\begin{itemize}
  \item [\textbf{STEP 2.}]
\end{itemize}
Notice that,
$$\frac{1}{6}=\frac{1}{4}-\sum_{k=2}^{\infty}\frac{1}{4^k}=\|x_1\|-\sum_{k=2}^{\infty}\|x_k\|\leq\|\sum_{k=1}^{\infty}x_k\|\leq\sum_{k=1}^{\infty}\|x_k\|=\sum_{k=1}^{\infty}\frac{1}{4^k}=\frac{1}{3}.$$
It follows that $y:=\sum_{k=1}^{\infty}x_k\in X$ with $\|y\|\in (\frac{1}{6},\frac{1}{3})$.
We next will show that $k_y=1$.
Since $\|(\lambda_n-T)^{-1}\|>1$, we have $\ln\|(\lambda_n-T)^{-1}\|>0$ for every $n$.
Then
\begin{align*}
&\frac{\ln\|(\lambda_n-T)^{-1}y\|}{\ln\|(\lambda_n-T)^{-1}\|}
=
\frac{\ln\|\sum\limits_{k=1}^{\infty}(\lambda_n-T)^{-1}x_k\|}{\ln\|(\lambda_n-T)^{-1}\|}\\
&\geq
\frac{\ln\left[\|(\lambda_n-T)^{-1}x_n\|-\left(\|(\lambda_n-T)^{-1}\sum\limits_{k=1}^{n-1}x_k\|+\sum\limits_{k=n+1}^{\infty}\|(\lambda_n-T)^{-1}x_k\|\right)\right]}
{\ln\|(\lambda_n-T)^{-1}\|}\\
&\geq
\frac{\ln\left[\|(\lambda_n-T)^{-1}x_n\|-\left(\|(\lambda_n-T)^{-1}\sum\limits_{k=1}^{n-1}x_k\|+\|(\lambda_n-T)^{-1}\|\sum\limits_{k=n+1}^{\infty}\frac{1}{4^k}\right)\right]}
{\ln\|(\lambda_n-T)^{-1}\|}\\
&=
\frac{\ln\left[\|(\lambda_n-T)^{-1}x_n\|-\left(\|(\lambda_n-T)^{-1}y_{n-1}\|+
\frac{1}{3\cdot4^n}\|(\lambda_n-T)^{-1}\|\right)\right]}
{\ln\|(\lambda_n-T)^{-1}\|}\\
&\geq
\frac{\ln\left[\|(\lambda_n-T)^{-1}x_n\|-\left(\|(\lambda_n-T)^{-1}\|^q+\frac{1}{3\cdot4^n}\|(\lambda_n-T)^{-1}\|\right)\right]}
{\ln\|(\lambda_n-T)^{-1}\|}\\
&=
\frac{\ln \left[s_n\|(\lambda_n-T)^{-1}e_n\|-\left(\|(\lambda_n-T)^{-1}\|^q+\frac{1}{3\cdot4^n}\|(\lambda_n-T)^{-1}\|\right)\right]}
  {\ln\|(\lambda_n-T)^{-1}\|}\\
&\geq
\frac{\ln \left[s_n(1-\frac{1}{n!4^n})\|(\lambda_n-T)^{-1}\|-\left(\|(\lambda_n-T)^{-1}\|^q+\frac{1}{3\cdot4^n}\|(\lambda_n-T)^{-1}\|\right)\right]}
{\ln\|(\lambda_n-T)^{-1}\|}\\
&=
\frac{\ln \left[\|(\lambda_n-T)^{-1}\|\left(s_n(1-\frac{1}{n!4^n})-\|(\lambda_n-T)^{-1}\|^{q-1}-\frac{1}{3\cdot4^n}\right)\right]}
{\ln\|(\lambda_n-T)^{-1}\|}\\
&=
\frac{\ln \left[\|(\lambda_n-T)^{-1}\|+\ln\left(s_n(1-\frac{1}{n!4^n})-\|(\lambda_n-T)^{-1}\|^{q-1}-\frac{1}{3\cdot4^n}\right)\right]}
{\ln\|(\lambda_n-T)^{-1}\|}\\
&=
1+\frac{\ln\left[\frac{1}{4^n}(1-\frac{1}{n!4^n}-\frac{1}{3})-\|(\lambda_n-T)^{-1}\|^{q-1}\right]}
{\ln\|(\lambda_n-T)^{-1}\|}\\
&=1+\frac{\ln\left[\frac{1}{4^n}\left((1-\frac{1}{n!4^n}-\frac{1}{3})-4^n\|(\lambda_n-T)^{-1}\|^{q-1}\right)\right]}
{\ln\|(\lambda_n-T)^{-1}\|}\\
&=1+\frac{\ln\frac{1}{4^n}+\ln\left[(1-\frac{1}{n!4^n}-\frac{1}{3})-4^n\|(\lambda_n-T)^{-1}\|^{q-1}\right]}
{\ln\|(\lambda_n-T)^{-1}\|}.
\end{align*}

Since $\|(\lambda_n-T)^{-1}\|>n^n$ and $q-1<0$, we have
$4^n\|(\lambda_n-T)^{-1}\|^{q-1}=0$ when $n\rightarrow \infty$.
So, $k_y\geq1$. And hence $k_y=1$.
\end{proof}

\section{Banach spaces with uniform multiplicity infinity}
In this section, we will discuss some basic properties of an infinite-dimensional Banach space with uniform multiplicity infinity.

Let $\mathcal{A}\subset \B(X)$ be an algebra, denote
$\mathcal{A}^{\prime}:=\{T\in\B(X):ST=TS
~~\text{for all}~~S\in \mathcal{A}\}$.
Let $n\in \mathbb{N}$ or $n=\aleph_0$, the countable cardinal number.
Recall that an algebra $\mathcal{A}\subset \B(X)$ has \emph{uniform multiplicity $n$} if $\mathcal{A}^{\prime}$ contains a system of matrix units $\{u_{i,j}\}_{i,j=0}^{n-1}$, where $\{u_{i,j}\}_{i,j=0}^{n-1}$ satisfied that
\begin{enumerate}
  \item [(1)]
  $\sup\{\|u_{j,k}\|:0\leq j,k<n\}<\infty$;
  \item [(2)]
  $u_{j,k}u_{l,m}=\delta_{kl} u_{j,m}$ for any $j,k,l,m$, where
  $$\delta_{kl}=\begin{cases}0, &
  k\neq l;\\
  1, & k=l;\end{cases}$$
  \item [(3)]
  $x=\sum_{k=0}^{n-1}u_{k,k}x$ for every $x\in X$, with convergence in norm.
\end{enumerate}

It is well known that the structure of an abelian von Neumann algebra can be described by uniform multiplicity theory
(see \cite[Theorem II.3.4]{{Davidson1996}}).
Inspired by this, we introduce the concept that an infinite-dimensional Banach space $X$ has uniform multiplicity infinity, which is presented in Definition \ref{I:uniform}.
Moreover, if $\B(X)$ contains a system of matrix units $\{u_{i,j}\}_{i,j=0}^{\infty}$ and $\dim \ran u_{0,0}=n<\infty$,
then we called $X$ has \emph{uniform multiplicity infinity of finite index $n$}.

Recall that a non-zero sequence $\{x_n\}_{n=0}^{\infty}$ in a Banach space $X$ is called a \emph{(Schauder) basis }if, for each $x\in X$, there is a unique
sequence of scalars $\{a_n\}_{n=0}^{\infty}$
such that $x=\sum_{n=0}^{\infty}a_nx_n$, where the series converges in norm to $x$.

\begin{lem}[see {\cite[Theorem 3.1]{CNL2005}}]
Let $\{x_n\}_{n=0}^{\infty}$ be a basis for $X$.
Set $P_n: X\rightarrow X$,
$P_n(\sum_{k=0}^{\infty}a_k x_k)=\sum_{k=0}^{n}a_k x_k$.
Then every $P_n$ is continuous, and $K=\sup_n\|P_n\|<\infty$.
\end{lem}
In this case, the number $K=\sup_n\|P_n\|$ is called the {\em basis constant} of the basis $\{x_n\}_{n=0}^{\infty}$.
In fact, it follows from the proof of
Theorem 3.1 in \cite{CNL2005} that any Banach space with a basis can always be given an equivalent norm under which the basis constant is $1$.

Next, we will establish some fundamental properties of a Banach space with uniform multiplicity infinity.

\begin{prop}\label{B:uniP1}
Let $X$ be an infinite-dimensional Banach space.
Then $X$ has uniform multiplicity infinity of index $1$ if and only if $X$ has a basis.
\end{prop}
\begin{proof}
``$\Rightarrow$".
Suppose that $\{u_{i,j}\}_{i,j=0}^{\infty}$ is a system of matrix units of $\B(X)$ and $\dim \ran u_{0,0}=1$.
Since $\dim \ran u_{0,0}=1$, there exists a $x_0\in\ran u_{0,0}$ such that $\|x_0\|=1$ and $\ran u_{0,0}=\mathbb{C}x_0$.
For every $j\geq 0$, let $x_j=u_{j,0}x_0$.
Since $u_{j,j}=u_{j,0}u_{0,0}u_{0,j}$ and $u_{0,0}x_0=x_0$, we have
$x_j\in \ran u_{j,j}$ and
$u_{j,j} x_j=u_{j,j}u_{j,0}x_0=u_{j,0}x_0=x_j$.
Hence, $\ran u_{j,j}=\mathbb{C}x_j$.
Moreover,
since $x=\sum_{k=0}^{\infty} u_{k,k}x$ for any $x\in X$ and
$u_{k,k}x\in\ran u_{k,k}=\mathbb{C}x_k$ for each $k\geq 0$,
there is a scalar $a_k$ such that $u_{k,k}x=a_kx_k$.
Thus, $x=\sum_{k=0}^{\infty}a_kx_k$, where the series converges in norm to $x$.
In particular, if $x=\sum_{k=0}^{\infty}b_kx_k$, then
$$u_{j,j}x=u_{j,j}\sum\limits_{k=0}^{\infty}b_kx_k=\sum\limits_{k=0}^{\infty}b_k u_{j,j}u_{k,0}x_0=b_j u_{j,0}x_0=b_jx_j.$$
Since $u_{j,j}x=a_jx_j$, we have $a_j=b_j$ for every $j$.
Therefore, this representation is unique.
In summary, $X$ has a basis.

``$\Leftarrow$".
Let $\{e_n\}_{n=0}^{\infty}$ be a basis for $X$ and $\|e_n\|=1$ for each
$n$.
For $x=\sum_{n=0}^{\infty}\alpha_n e_n\in X$,
let $f_n(x)=\alpha_{n}$. Then $f_n\in X^\prime$, where $X^\prime$ denote the dual of $X$.
In fact,
if we suppose that basis constant of the basis $\{e_n\}_{n=0}^{\infty}$ is $\frac{K}{2}$, then it is not difficult to check that $\|f_n\|\leq K$ for every $n$.
For $k,j\geq 0$, let $u_{k,j}(x)=f_j(x)e_k$ for any $x\in X$.
One can verify that $u_{k,j}$ is a linear operator in $X$ for every $j,k\geq 0$.
Moreover,
since
$$\|u_{k,j}(x)\|=\|f_j(x)e_k\|\leq \|f_j\|\|x\|\leq K \|x\|,$$
we have $\|u_{j,k}\|\leq K$ for $j,k\geq 0$.
Hence, $\{u_{j,k}\}_{j,k=0}^{\infty}\subset \B(X)$ and $\sup\{\|u_{j,k}\|:j,k\geq 0\}\leq K$.
For any $j,k,l,m\geq 0$, since
$$u_{j,k}u_{l,m}(x)=u_{j,k}f_m(x)e_l=f_m(x)f_k(e_l)e_j
=\delta_{kl} u_{j,m}(x),$$
we have $u_{j,k}u_{l,m}=\delta_{kl} u_{j,m}$.
Meanwhile,
we have
$x=\sum_{k=0}^{\infty}f_k(x)e_k=\sum_{k=0}^{\infty} u_{k,k}x$ for any $x\in X$.
Therefore, $\{u_{i,j}\}_{i,j=0}^{\infty}$ is a system of matrix units of $\B(X)$.
Since $\ran u_{0,0}=\mathbb{C} e_0$, we have $\dim\ran u_{0,0}=1$.
Hence, $X$ has uniform multiplicity infinity of index $1$.
\end{proof}

Recall that
a basis $\{x_n\}_{n=0}^{\infty}$ of a Banach space $X$ is said to be \emph{unconditional} if every convergent series of the form $\sum_n a_nx_n$,
converges unconditionally.
The simplest examples of unconditional bases are the spaces
$c_0$, $\ell^p(1\leq p<\infty)$ and $L^p[0,1](1<p<\infty)$.
However,
the spaces $C[0,1]$ (see\cite{Karlin1948}) or
$L^1[0,1]$ (see\cite{Pelczynski1961}) are spaces which fail to have an unconditional basis.
From Proposition \ref{B:uniP1}, it can be seen that an infinite-dimensional Banach space with uniform multiplicity infinity of index $1$ may not have an unconditional basis.

\begin{prop}\label{S:equi}
Let $X$ be an infinite-dimensional Banach space. Then
$X$ has uniform multiplicity infinity of finite index $n$ if and only if
$X$ has uniform multiplicity infinity of index $1$.
\end{prop}
\begin{proof}
``$\Rightarrow$".
Suppose that $\{u_{i,j}\}_{i,j=0}^{\infty}$ is a system of matrix units of $\B(X)$ and $\dim\ran u_{0,0}=n<\infty$.
From Proposition \ref{B:uniP1}, we need only prove that $X$ has a basis.
Since $\dim\ran u_{0,0}=n$,
we have $\dim\ran u_{i,i}=n$ for each $i\geq 0$.
Let $\{f_0,f_1,\ldots,f_{n-1}\}$ be a basis of $\ran u_{0,0}$.
For each $i\geq 0$ and $t=0,1,\ldots,n-1$, let
$e_{ni+t}=u_{i,0}f_t$.
Then
$\{e_{ni},e_{ni+1},\ldots,e_{ni+n-1}\}$
be a basis of $\ran u_{i,i}$.
We will show that $\{e_{j}\}_{j=0}^{\infty}$ is a basis for $X$.

Since $x=\sum_{k=0}^{\infty}u_{k,k}x$ for any $x\in X$ and $u_{k,k}x\in
\ran u_{k,k}$ for every $k\geq 0$, there is a unique
sequence of scalars $\{a_{kn},a_{kn+1},\ldots,a_{kn+n-1}\}$
such that $u_{k,k}x=\sum_{t=0}^{n-1}a_{kn+t}e_{kn+t}$.
It follows that
$$x=\sum_{k=0}^{\infty}\sum_{t=0}^{n-1}a_{kn+t}e_{kn+t}=\sum_{j=0}^{\infty}a_je_j,$$
where $j=kn+t$.
Since
$x=\sum_{k=0}^{\infty} u_{k,k} x$ converges in norm, it follows that
for any
$\varepsilon>0$, there exists a positive integer $K_0$ such that
$\|\sum_{k=K}^{\infty}u_{k,k}x\| \leq \frac{\varepsilon}{2}$ when $K\geq K_0$.
Let $S_K=\sum_{k=0}^{K}u_{k,k}x=\sum_{j=0}^{n(K+1)-1}a_je_j$,
then $\|x-S_{K}\|<\frac{\varepsilon}{2}$ for all $K\geq K_0$.
For any $J\geq nK_0$, let $k=\lfloor \frac{J}{n} \rfloor$, then $k\geq K_0$.
Set $l=J-nK$ $(0\leq l<n)$, then
$$\sum_{j=0}^{J}a_je_j=S_k+\sum_{t=0}^{l}a_{kn+t}e_{kn+t}.$$
Since $u_{k,k}x=\sum_{t=0}^{n-1}a_{kn+t}e_{kn+t}$, there exists a constant $D$ such that
$$\|\sum_{t=0}^{l}a_{kn+t}e_{kn+t}\|\leq D\|u_{k,k} x\|.$$
Since $\lim_{k\to\infty}\|u_{k,k}x\|=0$, there exists a positive integer $K_1$ such that
$\|u_{k,k} x\|<\frac{\varepsilon}{2D}$.
Let $K=\max\{K_0,K_1\}$, then
$$\|\sum_{j=0}^{J}a_j e_j-S_k\|<D\cdot \frac{\varepsilon}{2D}=\frac{\varepsilon}{2}$$
when $J\geq nK$ and $k\geq K$.
Therefore,
$$\|x-\sum_{j=0}^{J}a_je_j\|\leq \|x-S_k\|+\|S_k-\sum_{j=0}^{J}a_je_j\|<\frac{\varepsilon}{2}+\frac{\varepsilon}{2}=\varepsilon.$$
This shows that
the series
$\sum_{j=0}^{\infty} a_j e_j$ converges to $x$ in norm.
To prove uniqueness,
let $x=\sum_{j=0}^{\infty}b_je_j$.
Then
$$u_{k,k} x=u_{k,k}\sum_{j=0}^{\infty}b_j e_j=\sum_{t=0}^{n-1}b_{kn+t}e_{kn+t}.$$
Since $u_{k,k}x=\sum_{t=0}^{n-1}a_{kn+t}e_{kn+t}$ and $\{e_{kn+t}\}_{t=0}^{n-1}$ is a basis of $\ran u_{k,k}$,
we have $a_{kn+t}=b_{kn+t}$ for all $t$.
In summary, $\{e_j\}_{j=0}^{\infty}$ is a basis of $X$.

``$\Leftarrow$".
Let $\{u_{i,j}\}_{i,j=0}^{\infty}$ be a system of matrix units of $\B(X)$ and $\dim\ran u_{0,0}=1$.
Next, we show that $X$ has uniform multiplicity infinity of finite index $n$.
For any $k,l\geq 0$, let
$$v_{k,l}=\sum_{r=0}^{n-1}u_{nk+r,nl+r}.$$
Denote $C=\sup\{\|u_{k,l}\|:k,l\geq 0\}<\infty$, then $\|v_{k,l}\|\leq\sum_{r=0}^{n-1}\|u_{nk+r,nl+r}\|\leq nC$.
So, $\sup\{\|v_{k,l}\|:k,l\geq 0\}\leq nC<\infty$.
For any $k,l,p,q\geq 0$,
\begin{align*}
v_{k,l}v_{p,q}=\left(\sum_{r=0}^{n-1}u_{nk+r,nl+r}\right)\left(\sum_{s=0}^{n-1}u_{np+s,nq+s}\right)
&=\sum_{r,s=0}^{n-1}u_{nk+r,nl+r}u_{np+s,nq+s}=\delta_{lp}v_{k,q}.
\end{align*}
For any $x\in X$,
we have
$$x=\sum_{j=0}^{\infty}u_{j,j}x=\sum_{k=0}^{\infty}\sum_{r=0}^{n-1}u_{nk+r,nk+r}x=\sum_{k=0}^{\infty}v_{k,k}x.$$
Hence,
$\{v_{k,l}\}_{k,l=0}^{\infty}$ is a system of matrix units of $\B(X)$.
In particular,
since $\ran v_{0,0}=\ran \left(\sum_{r=0}^{n-1}\ran u_{r,r}\right)$ and $\dim\ran u_{r,r}=1$ for every $r$, we have $\dim\ran v_{0,0}=n$.
In summary, $X$ has uniform multiplicity infinity of index $n$.
\end{proof}

Proposition \ref{B:uniP1} and Proposition \ref{S:equi} show that an infinite-dimensional Banach space has a basis if and only if it has uniform multiplicity infinity of finite index.
However, the case of an infinite-dimensional Banach space with uniform multiplicity infinity of infinite index is different from that of finite index case.
In the following sections, we will prove that some classical Banach spaces with uniform multiplicity infinity of infinite index.

Let $\{X_n\}_{n=0}^{\infty}$ be a sequence of Banach spaces, and let $\|\cdot\|$ denote the norm on each $X_n$.
Define
$$\oplus_{0} X_n=\{x=\{x_n\}_{n=0}^{\infty}\in\prod_{n=0}^{\infty}X_n:\lim_{n\rightarrow \infty }\|x_n\|=0\},$$
endowed with the norm $\|x\|=\sup_{n}\|x_n\|$.

For $1\leq p<\infty$, define
$$\oplus_{p} X_n=\{x=\{x_n\}_{n=0}^{\infty}\in\prod_{n=0}^{\infty}X_n:\|x\|_p=\big(\sum_{n=0}^{\infty}\|x_n\|^p\big)^{\frac{1}{p}}<\infty\}.$$

\begin{lem}\label{L:direct}
Let $\{X_i\}_{i=0}^{\infty}$ be a sequence of infinite-dimensional Banach spaces.
Suppose that there exists a family $\{T_{ij}\}_{i,j=0}^{\infty}$ of bounded linear operators, where each $T_{ij}:X_j\rightarrow X_i$
is a linear homeomorphism, and such that
$\sup\{\|T_{ij}\|,\|T_{ij}^{-1}\|:i,j\geq 0\}<\infty$.
Then $\oplus_p X_i$ for $1\leq p<\infty$ and $\oplus_0 X_i$ have uniform multiplicity infinite of infinite index.
\end{lem}
\begin{proof}
Let $C=\sup\{\|T_{ij}\|,\|T_{ij}^{-1}\|:i,j\geq 0\}<\infty$.
For each $i\geq 0$, let $\phi_i:X_i\rightarrow X_0$, $\phi_i=T_{0i}$.
Then each $\phi_{i}$ is a linear
homeomorphism with $\|\phi_{i}\|\leq C$ and $\|\phi^{-1}_{i}\|\leq C$.
For $i,j\geq 0$,
let $S_{ij}:X_j\rightarrow X_i$,
$S_{ij}=\phi^{-1}_{i}\circ\phi_j$.
Then each $S_{ij}$ is a linear
homeomorphism.
Notice that, for any $i,j,l\geq 0$, we have
$$S_{ij}S_{jl}=(\phi_{i}^{-1}\circ \phi_{j})\circ (\phi_{j}^{-1}\circ \phi_l)=\phi_{i}^{-1}\circ\phi_{l}=S_{il}.$$
In particular,
$\|S_{ij}\|\leq\|\phi_{i}^{-1}\|\|\phi_j\|\leq C^2$,
and $\|S_{ij}^{-1}\|=\|S_{ji}\|\leq C^2$.
Hence, $\sup\{\|S_{ij}\|,\|S_{ij}^{-1}\|:i,j\geq 0\}\leq C^2$.

Denote $X=\oplus_p X_i$ or $X=\oplus_0 X_i$.
For $i\geq 0$, let
$$P_i:X\rightarrow X_i,\quad P_i(\{x_k\}_{k=0}^{\infty})=x_i,$$
and $J_i:X_i\rightarrow X$, $J_i(x)=(0,\ldots,0,x,0,\ldots)$ with $x$ in the $i$-th position.
Then $\|J_i\|= 1$ and $\|P_i\|= 1$.
For $i,j\geq 0$,
let
$$u_{i,j}=J_{i}S_{ij}P_j: X\rightarrow X.$$
We next shall show that $\{u_{i,j}\}_{i,j=0}^{\infty}$ is a system of matrix units of $\B(X)$.
Since $\sup\{\|S_{ij}\|,\|S_{ij}^{-1}\|:i,j\geq 0\}\leq C^2$, we have
$$\|u_{i,j}\|\leq \|J_i\|\|S_{ij}\|\|P_j\|\leq \|S_{ij}\|\leq C^2.$$
Thus, $\sup\{\|u_{i,j}\|:i,j\geq 0\}\leq C^2$.
For any $i,j,k,l\geq 0$, we have
$$u_{i,j}u_{k,l}=(J_i S_{ij} P_j)(J_k S_{kl}P_l)=J_i S_{ij}(P_jJ_k)S_{kl}P_l.$$
Since $P_j J_k=0$ when $j\neq k$ and $P_j J_k=id_{X_j}$ when $j=k$, we have $P_j J_k=\delta_{jk}id_{X_j}$.
Thus,
$$u_{i,j}u_{k,l}=\delta_{jk}J_iS_{ij}S_{kl}P_l=\delta_{jk}J_iS_{il}P_l=\delta_{jk}u_{i,l}.$$
Finally, for any $x=\{x_n\}_{n=0}^{\infty}\in X$, we have
$$\sum_{k=0}^{N}u_{k,k}(x)=\sum_{k=0}^{N}J_kS_{kk}P_k(x)=\sum_{k=0}^{N}J_k x_k=(x_0,\ldots,x_N,0,\ldots).$$
In the case of $X=\oplus_p X_n$,
$$\|x-\sum_{n=0}^{N}u_{n,n}x\|_p^p=\sum_{n=N+1}^{\infty}\|x_n\|^p\rightarrow 0$$
as $N\rightarrow \infty$.
In the case of $X=\oplus_0 X_n$,
$$\|x-\sum_{n=0}^{N}u_{n,n}x\|=\sup_{n>N}\|x_n\|\rightarrow 0$$
as $N\rightarrow \infty$.
So, $\{u_{i,j}\}_{i,j=0}^{\infty}$ is a system of matrix units of $\B(X)$.

Since $\ran u_{0,0}=J_0 (X_0)$ and
$\dim\ran X_0=\infty$, we have $\dim\ran u_{0,0}=\infty$.
Hence, $X$ has uniform multiplicity infinite of infinite index.
\end{proof}
It can be seen from Lemma \ref{L:direct} that a Banach space with uniform multiplicity infinite of infinite index may not have a basis.
Next, applying Lemma \ref{L:direct}, we show that some classical Banach spaces have uniform multiplicity infinity of infinite index.

\begin{prop}\label{B:lp}
For $1\leq p<\infty$, let
$$\ell^p=\{x=\{\alpha_n\}_{n=0}^{\infty}:\|x\|_p^p=\sum_{n=0}^{\infty}|\alpha_n|^p<\infty\}.$$
Then $\ell^p$ is a Banach space with uniform multiplicity infinity of infinity index.
\end{prop}
\begin{proof}
For $k=0,1,2,\cdots$, let
$$A_k=\{n\in \mathbb{N}:n=2^k(2m+1)-1,m=0,1,2,\cdots\}.$$
Clearly, the cardinal number $|A_k|=\infty$ for each $k$,  $\bigcup_{k=0}^{\infty} A_k=\bN$ and $A_k \cap A_j=\varnothing$ for $k\neq j$.
For each $k\geq 0$, let
$$X_k=\{x\in\ell^p:x=\sum_{n\in A_k}\alpha_ne_n,\alpha_n\in\bC\},$$
where $\{e_n\}_{n=0}^{\infty}$ is a canonical basis for $\ell^p$.
Then $\dim X_k=\infty$ for every $k$, and $\ell^p=\bigoplus_{p}X_k$.
To see this, let $\sigma_k:\bN\rightarrow A_k$, $\sigma_k(m)=2^k(2m+1)-1$,
then $\sigma_k$ is a bijection for each $k$.
Let
$$\phi_k:\ell^p\rightarrow X_k,\quad \phi_k(\{\alpha_m\}_{m=0}^{\infty})=\sum_{m=0}^{\infty}\alpha_me_{\sigma_k(m)}.$$
Then $\phi_k$ is an isometric isomorphism for every $k$.
Hence, $X_k\cong \ell^p$.
Moreover, for any $x=\{\alpha_n\}_{n=0}^{\infty}\in\ell^p$, we have
$$\|x\|_p^p=\sum_{n=0}^{\infty}|\alpha_n|^p=\sum_{k=0}^{\infty}\sum_{n\in A_k}|\alpha_n|^p=\sum_{k=0}^{\infty}\|x^{(k)}\|_p^p,$$
where $x^{(k)}=\{\alpha_n\}_{n\in A_k}$. Thus, $\ell^p=\bigoplus_{p}X_k$.

For $i,j\geq 0$, let $T_{ij}:X_j\rightarrow X_i$,
$$T_{ij}(\{\alpha_n\}_{n\in A_j})=\sum_{n\in A_j}\alpha_n e_{\tau_{ij}(n)},$$
where $\tau_{ij}:A_j\rightarrow A_i$, $\tau_{ij}=\sigma_i\circ\sigma_j^{-1}$.
Then
$T_{ij}$ is a linear homeomorphisms and $\sup\{\|T_{ij}\|,\|T_{ij}^{-1}\|:i,j\geq 0\}=1$.
Finally,
by Lemma \ref{L:direct}, $\ell^p$ has uniform multiplicity infinity of infinite index.
\end{proof}

\begin{cor}
Let
$$c_0=\{x=\{\alpha_n\}_{n=0}^{\infty}:\lim_{n\rightarrow\infty}\alpha_n=0\}.$$
Then $c_0$ is a Banach space with uniform multiplicity infinity of infinity index.
\end{cor}
\begin{proof}
It can be proved by the method used in the the proof of Proposition \ref{B:lp}.
\end{proof}

\begin{prop}
Let $L_0[0,1]$ be of the collection of all (equivalence classes, under equality almost everywhere) Lebesgue measurable functions $f:[0,1]\rightarrow \mathbb{C}$.
For $1\leq p<\infty$, let
$$L^p[0,1]=\{f\in L_0[0,1]:\|f\|_p=\left(\int_{[0,1]}|f(x)|^p \mathrm{d}x\right)^{\frac{1}{p}}<\infty\}.$$
Then $L^p[0,1]$ is a Banach space with uniform multiplicity infinity of infinity index.
\end{prop}

\begin{proof}
For $n\in \bN$, let $E_n = [1-\frac{1}{2^n}, 1-\frac{1}{2^{n+1}})$.
Then $\bigcup_{n \in \bN} E_n=[0,1)$ and $E_n \cap E_m=\varnothing$ for $n\neq m$.
Since the Lebesgue measure of any singleton is zero, one can consider the space $L^p[0,1)$ replace $L^p[0,1]$.
For any $f \in L^p[0,1)$, denote $f_n=f|_{E_n}$.
Then we have $f=g \in L^p[0,1)$ if and only if $f_n=g_n \in L^p(E_n)$ for all $n\in \bN$.
Notice that, for each $f \in L^p[0,1)$, we have
$$\|f\|_p=\left(\int_{[0,1]}|f|^p dx\right)^{\frac{1}{p}}=\left(\sum\limits_{n=0}^{\infty}\int_{E_n}|f|^p dx\right)^{\frac{1}{p}}=\left(\sum\limits_{n=0}^{\infty}\| f_n \|_{L^p(E_n)}^p\right)^{\frac{1}{p}},$$
where $\| f_n \|_{L^p(E_n)} =(\int_{E_n}|f_n|^p dx)^{\frac{1}{p}}$.
Hence, $L^p[0,1)\cong\bigoplus_{p}L^p(E_n)$.

We now construct a family of mutual linear homeomorphisms between the spaces $L^p(E_n)$.
For $i,j\geq 0$,
let
$$\phi_{ij}:E_j\rightarrow E_i,~~\phi_{ij}(x)=a_{ij}x+b_{ij},$$
where $a_{ij}=2^{j-i}$, $b_{ij}=1-2^{j-i}$.
Then $\phi_{ij}$ is a bijection and $$\phi_{ij}^{-1}(y)=\frac{1}{a_{ij}}y-\frac{b_{ij}}{a_{ij}}.$$
Let $T_{ij}:L^p(E_j)\rightarrow L^p(E_i)$,
$$(T_{ij}f)(y)=a_{ij}^{-\frac{1}{p}}f(\phi_{ij}^{-1}(y)),\quad y\in E_i.$$
Then $T_{ij}$ is a linear homeomorphism and $\sup\{\|T_{ij}\|,\|T_{ij}^{-1}\|:i,j\geq 0\}=1$.
In fact,
$$\|T_{ij}f\|^p_{L^p(E_i)}=\int_{E_i}|(T_{ij}f)(y)|^p \mathrm{d}y=a_{ij}^{-1}\int_{E_i}|f(\phi_{ij}^{-1}(y))|^p \mathrm{d}y.$$
Let $x=\phi_{ij}^{-1}(y)$, then $y=\phi_{ij}(x)$ and $\mathrm{d}y=a_{ij}\mathrm{d}x$. Thus
$$\|T_{ij}f\|^p_{L^p(E_i)}=a_{ij}^{-1}\int_{E_j}|f(x)|^p a_{ij}\mathrm{d}x=\int_{E_j}|f(x)|^p\mathrm{d}x=\|f\|_{L^p(E_j)}^p.$$
So, $\|T_{ij}\|=1$ for every $i,j\geq 0$.
Moreover, we have $T_{ij}^{-1}=T_{ji}$ and $\|T_{ji}\|=1$.
By Lemma \ref{L:direct}, $L^p[0,1]$ has uniform multiplicity infinity of infinite index.
\end{proof}

\begin{prop}\label{S3:cuni}
Let $C(\mathbb{T})$ be the Banach space of all continuous functions $f:\mathbb{T}\rightarrow \mathbb{C}$ with the norm $\|f\|=\sup\{|f(z)|:z\in\mathbb{T}\}$, where $\mathbb{T}$ denote unit circle in the complex plane.
Then $C(\mathbb{T})$ has uniform multiplicity infinity of infinity index.
\end{prop}
To give the proof of Proposition \ref{S3:cuni}, we first present the following lemma.

\begin{lem}\label{S3:Duni}
Let $X_1$ and $X_2$ be Banach spaces that both have uniform multiplicity infinity of infinity index.
Then the space $X_1\oplus X_2$, endowed with the maximum norm, is a Banach space with uniform multiplicity infinity of infinity index.
\end{lem}
\begin{proof}
Let $\{u_{i,j}\}_{i,j=0}^{\infty}$ and $\{v_{i,j}\}_{i,j=0}^{\infty}$ be a system of matrix units of $\B(X_1)$ and $\B(X_2)$, respectively.
Set $X=X_1\oplus X_2$.
For $i,j\geq 0$, let
$$w_{i,j}(x,y)=(u_{i,j}(x),v_{i,j}(y)),\quad (x,y)\in X.$$
It is not difficult to check that $\{w_{i,j}\}_{i,j=0}^{\infty}\subset \B(X)$.
We next prove that $\{w_{i,j}\}_{i,j=0}^{\infty}$ is a system of matrix units of $\B(X)$.
Suppose that there exist constants $C_1>0$ and $C_2>0$ such that
$\sup\{\|u_{i,j}\|:i,j\geq 0\}\leq C_1$ and $\sup\{\|v_{i,j}\|:i,j\geq 0\}\leq C_2$.
Let $C=\max\{C_1,C_2\}$.
For any $(x_1,x_2)\in X$, we have
\begin{align*}
\|w_{i,j}(x_1,x_2)\|_{X}
&=\max\{\|u_{i,j}(x_1)\|_{X_1},\|v_{i,j}(x_2)\|_{X_2}\}\\
&\leq\max\{\|u_{i,j}\|\|x_1\|_{X_1},\|v_{i,j}\|\|x_2\|_{X_2}\}\\
&\leq C\max\{\|x_1\|_{X_1},\|x_1\|_{X_2}\}=C\|(x_1,x_2)\|.
\end{align*}
It follows that $\sup\{\|w_{i,j}\|:i,j\geq 0\}\leq C$.

For any $i,j,k,l\geq 0$ and $(x_1,x_2)\in X$, we have
\begin{align*}
(w_{i,j}w_{k,l})(x_1,x_2)
&=w_{i,j}(w_{k,l}(x_1,x_2))\\
&=w_{i,j}(u_{k,l}(x_1),v_{k,l}(x_2))\\
&=\left(u_{i,j}(u_{k,l}(x_1)),v_{i,j}(v_{k,l}(x_2))\right)\\
&=\left((u_{i,j}u_{k,l})(x_1),(v_{i,j}v_{k,l})(x_2)\right)
\end{align*}
Since $u_{i,j}u_{k,l}=\delta_{jk}u_{i,l}$ and $v_{i,j}v_{k,l}=\delta_{jk} v_{i,l}$, we have
$$(w_{i,j}w_{k,l})(x_1,x_2)=(\delta_{jk}u_{i,l}(x_1),\delta_{jk}v_{i,l}(x_2))=\delta_{jk}(u_{i,l}(x_1),v_{i,l}(x_2))=\delta_{jk}w_{i,l}(x_1,x_2).$$
It follows that $w_{i,j}w_{k,l}=\delta_{jk}w_{i,l}$.
Finally, we prove that $\sum_{j=0}^{\infty} w_{j,j}x=x$ for any $x\in X$, i.e.,
$$\lim_{N\rightarrow \infty}\|x-\sum_{j=0}^{N} w_{j,j}x\|=0.$$
Notice that, for any $x=(x_1,x_2)\in X$, we have
$$\sum_{j=0}^{\infty} w_{j,j}(x_1,x_2)=\left(\sum_{j=0}^{\infty}u_{j,j}x_1,\sum_{j=0}^{\infty}v_{j,j}x_2\right).$$
Since $\sum_{j=0}^{\infty}u_{j,j}x_1=x_1$ and $\sum_{j=0}^{\infty}v_{j,j}x_2=x_2$, we have for any $\varepsilon > 0$,
there exists a positive integer $N_1$ such that
$\|x_1-\sum_{j=0}^{N}u_{j,j}x_1\|\leq \varepsilon$
when $N\geq N_1$.
And that
there exists a positive integer $N_2$ such that
$\|x_2-\sum_{j=0}^{N}v_{j,j}x_2\|\leq \varepsilon$
when $N\geq N_2$.
Let $N_0=\max\{N_1,N_2\}$, then
$$\|x-\sum_{j=0}^{N}w_{j,j}x\|=\max\{\|x_1-\sum_{j=0}^{N}u_{j,j}x_1\|,\|x_2-\sum_{j=0}^{N}v_{j,j}x_2\|\}\leq \varepsilon$$
when $N\geq N_0$.
In summary, $\{w_{i,j}\}_{i,j=0}^{\infty}$ is a system of matrix units of $\B(X)$.
Moreover,
let
$\ran u_{0,0}\times \{0_{X_2}\}=\{(a,0):a\in \ran u_{0,0}\}$, it is easy to check that $\ran u_{0,0}\times \{0_{X_2}\}\cong \ran u_{0,0}$.
Since $\dim\ran u_{0,0}=\infty$, we have $\dim(\ran u_{0,0}\times \{0_{X_2}\})=\infty$.
Thus,
$$\dim\ran w_{0,0}\geq \dim(\ran u_{0,0}\times \{0_{X_2}\})=\infty.$$
It follows that $\dim\ran w_{0,0}=\infty$. Hence, $X_1\oplus X_2$ has uniform multiplicity infinity of infinity index.
\end{proof}
\begin{proof}[Proof of Proposition~\ref{S3:cuni}]
Let $\mathbb{T}=\{\mathrm{e}^{2\pi\mathrm{i}t}:t\in[0,1]\}$, where $\mathrm{e}^{2\pi\mathrm{i}\cdot0}=\mathrm{e}^{2\pi\mathrm{i}\cdot1}=1$.
Let $z_0=1$ and $z_n=\mathrm{e}^{2\pi\mathrm{i}t_n}$ for $n\geq 1$, where $t_n=1-\frac{1}{2^n}$.
Then $z_n\rightarrow z_0$ when $n\rightarrow \infty$.
Let $I_0=\{\mathrm{e}^{2\pi\mathrm{i}t}:t\in[0,t_1]\}$ and
$I_n=\{\mathrm{e}^{2\pi\mathrm{i}t}:t\in[t_n,t_{n+1}]\}$
for $n\geq 1$.
Then $\bigcup_{n=0}^{\infty}I_n=\mathbb{T}$.
Let
$$C_0(\mathbb{T})=\{f\in C(\mathbb{T}):f(z_0)=0\},$$ we have
$C(\mathbb{T})\cong C_0(\mathbb{T})\oplus \mathbb{C}$.
For $n\geq 1$,
let
$$C_{0}(I_n)=\{f\in C(I_n):f(z_n)=f(z_{n+1})=0\}.$$
We claim that $C_0(\mathbb{T})\cong \left(\sideset{}{_0}\bigoplus\limits_{n=0}^{\infty} C_0(I_n)\right)\oplus c_0$.
To prove this, let $\psi:C_0(\mathbb{T})\rightarrow \left(\sideset{}{_0}\bigoplus\limits_{n=0}^{\infty} C_0(I_n)\right)\oplus c_0$,
$$\psi(g)=\left((g_0,g_1,\cdots),\{g(z_n)\}_{n=1}^{\infty}\right),$$
where
$$g_n(e^{2\pi i t})=g(e^{2\pi i t})-\left(g(z_n)+\frac{t-t_n}{t_{n+1}-t_n}\left(g(z_{n+1})-g(z_n)\right)\right),$$
for any $t\in [t_n,t_{n+1}]$.
Since
\begin{align*}
\|\psi(g)\|&=\max\{\sup_n\|g_n\|_{C_0(I_n)},\sup_n|g(z_n)|\}\\
&=\max\{\sup_n\sup_{z\in I_n}|g(z)|,\sup_n|g(z_n)|\}\leq 2\sup_{z\in\mathbb{T}}|g(z)|=2\|g\|
\end{align*}
and
\begin{align*}
\|g\|=\sup_n\max_{z\in I_n}|g(z)|
\leq&\sup_n\{\|g_n\|_{C_0(I_n)}+\max\{|g(z_n)|,|g(z_{n+1})|\}\}\\
\leq&\sup_n \|g_n\|_{C_0(I_n)}+\sup_n|g(z_n)|\leq 2\|\psi(g)\|,
\end{align*}
it follows that
$\frac{1}{2}\|g\|\leq\|\psi(g)\|\leq 2\|g\|$.
Therefore, $\psi$ is a linear homeomorphism.

Since $c_0\oplus \mathbb{C} \cong c_0$, we have $C(\mathbb{T})\cong \left(\sideset{}{_0}\bigoplus\limits_{n=0}^{\infty} C_0(I_n)\right)\oplus c_0$.
By Lemma \ref{L:direct}, we have $\sideset{}{_0}\bigoplus\limits_{n=0}^{\infty} C_0(I_n)$ has uniform multiplicity infinity of infinity index.
By Lemma \ref{S3:Duni}, we have $C(\mathbb{T})$ has uniform multiplicity infinity of infinity index.
\end{proof}

\section{Attainability of right closed subset of $[0,1]$}
In this section, we mainly present the proof of Theorem \ref{dlB}.
The following lemmas will be applied in later part.

\begin{lem}[see \cite{CNL2005}]\label{unc}
If $\{x_n\}_{n=0}^{\infty}$ is an unconditional basis of Banach space $X$ with basis constant $K$, then
for any $\xi=\{\lambda_j\}_{j=0}^{\infty}\in \ell^{\infty}$ and $x=\sum_{n=0}^{\infty}a_nx_n\in X$ have
$$\|M_{\xi}x\|\leq K\|\xi\|_{\infty}\|x\|,$$
where $M_{\xi}x=\sum_{j=0}^{\infty}\lambda_j a_jx_j$.
\end{lem}
Recall that the {\em right closure} $\lceil K\rceil$ of a nonempty set $K$ of $\mathbb{R}$ is defined as $\{\sup\sigma:\emptyset\neq\sigma\subset K, \sigma \ \text{is bounded} \}$.
\begin{lem}[see {\cite[Lemma 3.7]{JiZhang}}]\label{rcl}
If $\sigma$ is a nonempty right closed subset of $\mathbb{R}$, then there exists a sequence $\{r_n:n\geq1\}$ such that whose right closure is $\sigma$.
\end{lem}

To prove Theorem \ref{dlB}, we need make the following preparations.

For each right closed subset $\sigma$ of $[0,1]$ containing $1$,
by Lemma \ref{rcl}, one can pick a sequence $\{r_k:k\geq0\}$ so that $r_0=1$ and that the right closure of $\{\sqrt{r_k}: k\geq 0\}$ is $\sigma$.
Let $\{u_{i,j}\}_{i,j=0}^{\infty}$ be a system of matrix units of $\B(X)$.
Denote $X_0=\ran u_{0,0}$.
Suppose that $\{e_n^{(0)}\}_{n=0}^{\infty}$ is an unconditional basis with basis constant $K$ for subspace $X_0$.
For $x=\sum_{n=0}^{\infty}a_ne_n^{(0)}\in X_0$, define
$$A_0x=\sum_{n=0}^{\infty}r_n a_ne_n^{(0)}.$$
By Lemma \ref{unc}, we have $\|A_0x\|\leq K\|x\|$.
Since any Banach space with a basis can always be given an equivalent norm under which the basis constant is $1$, and the operators similar to each other have the same power set, one may assume that $K\equiv 1$.
Hence, $A_0\in \B(X_0)$ and $\|A_0\|\leq 1$.
Moreover, let
\begin{equation}\label{T:Reality}
T=\sum\limits_{n=0}^{\infty}\frac{u_{n+1,0}  A_0 u_{0,n}}{(2n+1)(2n+2)}.
\end{equation}

\begin{lem}\label{T:quani}
Let $T$ be as in $\eqref{T:Reality}$ above, then $T\in\B(X)$ and $T$ is quasinilpotent.
\end{lem}
\begin{proof}
For any $x=\sum_{n=0}^{\infty}u_{n,n}x\in X$, we have
$$Tx=\sum\limits_{n=0}^{\infty}\frac{u_{n+1,0} A_0 u_{0,n}x}{(2n+1)(2n+2)}.$$
It is easy to check that $T$ is a linear operator on $X$.
Let $\sup_{j,k}\|u_{j,k}\|=M$, we have
$$\|u_{n+1, 0} A_0 u_{0,n}\|\leq\|u_{n+1, 0}\|\|A_0\|\|u_{0,n}\|\leq M^2.$$
Thus
\begin{equation*}
\|T\|\leq \sum\limits_{n=0}^{\infty}\frac{1}{(2n+1)(2n+2)}
\|u_{n+1, 0}\|\|A_0\|\|u_{0,n}\|
\leq M^2\sum\limits_{n=1}^{\infty}\frac{1}{2 n^2}=\frac{M^2\pi^2}{12}.
\end{equation*}
It follows that $T\in\B(X)$.
For any $k\geq 1$, we have
$$T^k x=\sum\limits_{n=0}^{\infty}\frac{(2n)!u_{n+k, 0}A_0^k u_{0,n}x}{(2n+2k)!}.$$
So
$$\|T^k x\|\leq \sum\limits_{n=0}^{\infty}\frac{(2n)!}{(2n+2k)!}\|u_{n+1, 0}\|\|A_0\|^k\|u_{0,n}\|\|x\|\leq M^2\sum\limits_{n=0}^{\infty}\frac{(2n)!}{(2n+2k)!}\|x\|.$$
Denote
$\alpha_k=M^2\sum_{n=0}^{\infty}\frac{(2n)!}{(2n+2k)!}$, then $\|T^k\|\leq \alpha_k$.
Notice that, for any $k\geq 1$,
$$\alpha_{k+1}=\sum\limits_{n=0}^{\infty}\frac{M^{2}(2n)!}{(2n+2k+2)!}
\leq \frac{1}{(2k+1)(2k+2)}\sum\limits_{n=0}^{\infty}\frac{M^{2}(2n)!}{(2n+2k)!}
= \frac{\alpha_{k}}{(2k+1)(2k+2)}.$$
Therefore, we have $\alpha_k\leq \frac{\alpha_1}{(2k)!}$ for any $k\geq 2$.
It follows that
\begin{align*}
\lim\limits_{k\rightarrow\infty}\|T^k\|^{\frac{1}{k}}
\leq\lim\limits_{k\rightarrow\infty}(\alpha_{k})^{\frac{1}{k}}
\leq\lim\limits_{k\rightarrow\infty}\left(\frac{\alpha_1}{(2k)!}\right)^{\frac{1}{k}}=0.
\end{align*}
Thus, $\lim_{k\rightarrow\infty}\|T^k\|^{\frac{1}{k}}=0$.
We have $r(T)=0$ and $T$ is quasinilpotent.
\end{proof}

\begin{lem}
Let $T$ be as in Lemma \ref{T:quani}, then
\begin{equation}\label{S4:lnorm}
\|(\lambda-T)^{-1}\|\leq \frac{1}{|\lambda|}+ \frac{M^2}{|\lambda|}e^{\frac{1}{\sqrt{|\lambda|}}}+
\frac{M^2}{(\sqrt{|\lambda|}+1)|\lambda|}e^{\frac{1}{\sqrt{|\lambda|}}}, \ \ \forall\lambda\neq 0,
\end{equation}
where $M=\sup\{\|u_{j,k}\|:j,k\geq 0\}$.
\end{lem}
\begin{proof}
For $\lambda \neq 0$, we have
$$\|(\lambda-T)^{-1}\|=
\| \sum\limits_{k=0}^{\infty}\frac{T^k}{\lambda^{k+1}}\|\leq\| \sum\limits_{k=1}^{\infty}\frac{T^k}{\lambda^{k+1}}\|+\frac{1}{|\lambda|}.$$
For $\|\sum\limits_{k=1}^{\infty}\frac{T^k}{\lambda^{k+1}}\|$, we have
\begin{align*}
\|\sum\limits_{k=1}^{\infty}\frac{T^k}{\lambda^{k+1}}\|
&=\|\sum\limits_{k=1}^{\infty}\sum\limits_{n=0}^{\infty}
\frac{(2n)!u_{n+k, 0}A_0^k u_{0,n}}{(2n+2k)!\lambda^{k+1}}\|
=\|\sum\limits_{n=0}^{\infty}\sum\limits_{k=1}^{\infty}
\frac{(2n)!u_{n+k, 0}A_0^k u_{0,n}}{(2n+2k)!\lambda^{k+1}}\|\\
&\leq\|\sum\limits_{n=1}^{\infty}\sum\limits_{k=1}^{\infty}
\frac{(2n)!u_{n+k, 0}A_0^k u_{0,n}}{(2n+2k)!\lambda^{k+1}}\|
+\|\sum\limits_{k=1}^{\infty}
\frac{u_{n+k, 0}A_0^k u_{0,n}}{(2k)!\sqrt \lambda^{2k+2}}\|\\
&\leq\|\sum\limits_{n=1}^{\infty}\sum\limits_{k=1}^{\infty}
\frac{(2n)!u_{n+k, 0}A_0^k u_{0,n}}{(2n+2k)!\lambda^{k+1}}\|+\frac{M^2}{|\lambda|} \sum\limits_{k=1}^{\infty}\frac{1}{(2k)!\sqrt{|\lambda|}^{2k}}\\
&\leq\|\sum\limits_{n=1}^{\infty}\sum\limits_{k=1}^{\infty}
\frac{(2n)!u_{n+k, 0}A_0^k u_{0,n}}{(2n+2k)!\lambda^{k+1}}\|
+\frac{M^2}{|\lambda|}e^{\frac{1}{\sqrt{|\lambda|}}}.
\end{align*}
For
$\|\sum\limits_{n=1}^{\infty}\sum\limits_{k=1}^{\infty}
\frac{(2n)!u_{n+k, 0}A_0^k u_{0,n}}{(2n+2k)!\lambda^{k+1}}\|$,
we have
\begin{align*}
&\|\sum\limits_{n=1}^{\infty}\sum\limits_{k=1}^{\infty}
\frac{(2n)!u_{n+k, 0}A_0^k u_{0,n}}{(2n+2k)!\lambda^{k+1}}\|
\leq M^2\sum\limits_{n=1}^{\infty}\sum\limits_{k=1}^{\infty}\frac{(2n)!}{(2n+2k)!|\lambda|^{k+1}}\\
&=
\frac {M^2}{(\sqrt{|\lambda|}+1)|\lambda|}
\sum\limits_{n=1}^{\infty}\sum\limits_{k=1}^{\infty}
\frac{(\sqrt{|\lambda|}+1)(2n)!}{(2n+2k)!|\lambda|^{k}}\\
&=
\frac{M^2}{(\sqrt{|\lambda|}+1)|\lambda|}
\sum\limits_{n=1}^{\infty}\sum\limits_{k=1}^{\infty}
\frac{\sqrt{|\lambda|}(2n)!}{(2n+2k)!|\lambda|^{k}}+\frac{(2n)!}{(2n+2k)!|\lambda|^{k}}\\
&=
\frac{M^2}{(\sqrt{|\lambda|}+1)|\lambda|}
\sum\limits_{n=1}^{\infty}\sum\limits_{k=1}^{\infty}\frac{(2n)!}{(2n+2k)!\sqrt{|\lambda|}^{2k-1}}+\frac{(2n)!}{(2n+2k)!\sqrt{|\lambda|}^{2k}}\\
&\leq
\frac{M^2}{(\sqrt{|\lambda|}+1)|\lambda|}
\sum\limits_{n=1}^{\infty}\sum\limits_{k=1}^{\infty}\left(\frac{1}{(2k)!\sqrt{|\lambda|}^{2k-1}}+\frac{1}{(2k)!\sqrt{|\lambda|}^{2k}}\right)\left(\prod\limits_{j=1}^{2n}\frac{j}{2k+j}\right).
\end{align*}
For $n,k\geq 1$, denote $P(n,k)=\prod\limits_{j=1}^{2n}\frac{j}{2k+j}$.
Fixed $n$, we have
\begin{align*}
P(n,k)\leq P(n,1)
&=\prod\limits_{j=1}^{2n}\frac{j}{2+j}
=\frac{1 \cdot 2 \cdot 3 \cdots (2n)}{3 \cdot 4 \cdot 5 \cdots (2n+2)}\\
&=\frac{2}{(2n+1)(2n+2)}.
\end{align*}
Hence,
\begin{align*}
&\|\sum\limits_{n=1}^{\infty}\sum\limits_{k=1}^{\infty}
\frac{(2n)!u_{n+k, 0}A_0^k u_{0,n}}{(2n+2k)!\lambda^{k+1}}\|\\
&\leq
\frac{M^2}{(\sqrt{|\lambda|}+1)|\lambda|}
\sum\limits_{n=1}^{\infty}\sum\limits_{k=1}^{\infty}\left(\frac{1}{(2k)!\sqrt{|\lambda|}^{2k-1}}+\frac{1}{(2k)!\sqrt{|\lambda|}^{2k}}\right)\left(\prod\limits_{j=1}^{2n}\frac{j}{2k+j}\right)\\
&\leq
\frac{M^2}{(\sqrt{|\lambda|}+1)|\lambda|}
\sum\limits_{n=1}^{\infty}\frac{2}{(2n+1)(2n+2)}\sum\limits_{k=1}^{\infty}\left(\frac{1}{(2k-1)!\sqrt{|\lambda|}^{2k-1}}+\frac{1}{(2k)!\sqrt{|\lambda|}^{2k}}\right)\\
&\leq
\frac{M^2}{(\sqrt{|\lambda|}+1)|\lambda|}
\sum\limits_{k=1}^{\infty}\frac{1}{k!\sqrt{|\lambda|}^{k}}\\
&\leq
\frac{M^2}{(\sqrt{|\lambda|}+1)|\lambda|}e^{\frac{1}{\sqrt{|\lambda|}}}.
\end{align*}
Therefore,
\begin{equation*}
\|(\lambda-T)^{-1}\|\leq \frac{1}{|\lambda|}+ \frac{M^2}{|\lambda|}e^{\frac{1}{\sqrt{|\lambda|}}}+
\frac{M^2}{(\sqrt{|\lambda|}+1)|\lambda|}e^{\frac{1}{\sqrt{|\lambda|}}},
\end{equation*}
which completes the proof.
\end{proof}

\begin{lem}\label{T:normvector}
Let $T$ be as in Lemma \ref{T:quani}.
For each $m\geq 0$, then there exists a constant $c>0$ such that
\begin{equation}\label{T:ej}
\|(|\lambda|-T)^{-1}e_j^{(m)}\|\geq
\frac{1}{c}\frac{(2m)!\sqrt{|\lambda|}|\lambda|^{m}}{\sqrt{r_j}(r_j^{m+1}+r_j^{m+1}\sqrt{|\lambda|})}
\left(e^{\frac{\sqrt{r_j}}{\sqrt{|\lambda|}}}-\sum_{i=0}^{2m+1}\frac{\sqrt{r_j}^i}{i!\sqrt{|\lambda|}^{i}}\right),
\end{equation}
where $e_j^{(m)}=u_{m,0}e_{j}^{(0)}$.
\end{lem}
\begin{proof}
For $n\geq 0$, let $\phi_n (\sum_{j=0}^{\infty}\alpha_j e_j^{(0)})=\alpha_n$ for any $\sum_{j=0}^{\infty}\alpha_j e_j^{(0)}\in X_0$.
Since the basis constant $K\equiv1$, we have $\|\phi_n\|\leq2$.
For each $n,j\geq 0$, let $f_n^{(j)}(x)=\phi_n(u_{0,j}x)$ for any $x\in X$.
Since $\|f_n^{(j)}(x)\|\leq \|\phi_n\|\|u_{0,j}\|\|x\|\leq 2M \|x\|$,
we have $\|f_n^{(j)}\|\leq 2M$.
Then $f_n^{(j)}\in X^{\prime}$ for every $n,j\geq 0$.
Moreover, for $m\geq 0$,
let
$$\varphi_j^{(m)}=\sum_{n=0}^{\infty}\frac{f_j^{(n+m)}}{(2n+1)(2n+2)}.$$
Since
$$\|\varphi_j^{(m)}\|\leq \sum_{n=0}^{\infty}\frac{1}{(2n+1)(2n+2)}\|f_j^{(n+m)}\|
\leq 2M\sum_{n=1}^{\infty}\frac{1}{2n^2}=\frac{M\pi^2}{6}$$
and $\varphi_j^{(m)}(e_{j}^{(m)})=\frac{1}{2}$, we have $\varphi_j^{(m)}\in X^{\prime}$ and $\varphi_j^{(m)}\neq 0$ for every $m,j\geq 0$.
Then there exists a positive number $c$ with $0<c\leq\frac{M\pi^2}{6}$ such that $c=\|\varphi_j^{(m)}\|$.
For $\lambda\neq 0$, we have
\begin{equation*}\label{T:norej}
(\lambda-T)^{-1}e_j^{(m)}
=\sum_{k=0}^{\infty}\sum_{n=0}^{\infty}\frac{(2n)!u_{k+n,0}A_0^ku_{0,n}e_j^{(m)}}{(2k+2n)!\lambda^{k+1}}
=\sum_{k=0}^{\infty}\frac{(2m)!r_j^k e_j^{(k+m)}}{(2k+2m)!\lambda^{k+1}}.
\end{equation*}
Since
\begin{align*}
&\varphi_j^{(m)}\left((|\lambda|-T)^{-1}e_j^{(m)}\right)
=\sum\limits_{k=0}^{\infty}\frac{(2m)!{r_j}^k}{(2k+2m+2)!|\lambda|^{k+1}}\\
&=\frac{(2m)!}{(r_j^{m+1}+r_j^{m+1}\sqrt{|\lambda|})}
\sum\limits_{k=0}^{\infty}\frac{(r_j^{m+1}+r_j^{m+1}\sqrt{|\lambda|}){r_j}^k}{(2k+2m+2)!|\lambda|^{k+1}}\\
&=\frac{(2m)!|\lambda|^{m}}{(r_j^{m+1}+r_j^{m+1}\sqrt{|\lambda|})}
\sum\limits_{k=0}^{\infty}\left(\frac{{r_j}^{k+1+m}}{(2k+2m+2)!|\lambda|^{k+1+m}}
+\frac{\sqrt{|\lambda|}{r_j}^{k+1+m}}{(2k+2m+2)!|\lambda|^{k+1+m}}\right)\\
&=\frac{(2m)!|\lambda|^{m}}{\sqrt{r_j}(r_j^{m+1}+r_j^{m+1}\sqrt{|\lambda|})}
\sum\limits_{k=0}^{\infty}\left(\frac{(\sqrt{{r_j}})^{2k+3+2m}}{(2k+2m+2)!(\sqrt{|\lambda|})^{2k+2+2m}}+
\frac{(\sqrt{{r_j}})^{2k+2+2m}}{(2k+2m+2)!(\sqrt{|\lambda|})^{2k+1+2m}}\right)\\
&=\frac{(2m)!\sqrt{|\lambda|}|\lambda|^{m}}{\sqrt{r_j}(r_j^{m+1}+r_j^{m+1}\sqrt{|\lambda|})}
\sum\limits_{k=0}^{\infty}\left(\frac{(\sqrt{{r_j}})^{2k+3+2m}}{(2k+2m+2)!(\sqrt{|\lambda|})^{2k+3+2m}}+
\frac{(\sqrt{{r_j}})^{2k+2+2m}}{(2k+2m+2)!(\sqrt{|\lambda|})^{2k+2+2m}}\right)\\
&\geq\frac{(2m)!\sqrt{|\lambda|}|\lambda|^{m}}{\sqrt{r_j}(r_j^{m+1}+r_j^{m+1}\sqrt{|\lambda|})}
\sum\limits_{k=0}^{\infty}\left(\frac{(\sqrt{{r_j}})^{2k+3+2m}}{(2k+2m+3)!(\sqrt{|\lambda|})^{2k+3+2m}}+
\frac{(\sqrt{{r_j}})^{2k+2+2m}}{(2k+2m+2)!(\sqrt{|\lambda|})^{2k+2+2m}}\right)\\
&=
\frac{(2m)!\sqrt{|\lambda|}|\lambda|^{m}}{\sqrt{r_j}(r_j^{m+1}+r_j^{m+1}\sqrt{|\lambda|})}
\left(e^{\frac{\sqrt{r_j}}{\sqrt{|\lambda|}}}-\sum_{i=0}^{2m+1}\frac{\sqrt{r_j}^i}{i!\sqrt{|\lambda|}^{i}}\right),
\end{align*}
it follows that
\begin{align*}
\|(|\lambda|-T)^{-1}e_j^{(m)}\|
&\geq
\frac{1}{\|\varphi_j^{(m)}\|}\left|\varphi_j\left((|\lambda|-T)^{-1}e_j^{(m)}\right)\right|\\
&\geq
\frac{1}{c}\frac{(2m)!\sqrt{|\lambda|}|\lambda|^{m}}{\sqrt{r_j}(r_j^{m+1}+r_j^{m+1}\sqrt{|\lambda|})}
\left(e^{\frac{\sqrt{r_j}}{\sqrt{|\lambda|}}}-\sum_{i=0}^{2m+1}\frac{\sqrt{r_j}^i}{i!\sqrt{|\lambda|}^{i}}\right).
\end{align*}
The proof is complete.
\end{proof}

\begin{lem}\label{T:normal}
Let $T$ be as in Lemma \ref{T:quani},
then
\begin{equation}\label{T:leqnorm}
\|(|\lambda|-T)^{-1}\|\geq \frac{1}{c}\frac{\sqrt{|\lambda|}}{1+\sqrt{|\lambda|}}
\left(e^{\frac{1}{\sqrt{|\lambda|}}}-1-\frac{1}{\sqrt{|\lambda|}}\right),\ \ \forall\lambda\neq 0.
\end{equation}
\end{lem}

\begin{proof}
Notice that,
$$\|(|\lambda|-T)^{-1}\|\geq\|(|\lambda|-T)^{-1}e_0^{(0)}\|
\geq
\frac{1}{\|\varphi_0^{(0)}\|}\left|\varphi_0^{(0)}\left((|\lambda|-T)^{-1}e_0^{(0)}\right)\right|.$$
For
$\varphi_0^{(0)}\left((|\lambda|-T)^{-1}e_0^{(0)}\right)$,
by Lemma \ref{T:normvector},
we have
\begin{align*}
\varphi_0^{(0)}\left((|\lambda|-T)^{-1}e_0^{(0)}\right)
\geq
\frac{\sqrt {|\lambda|}}{1+\sqrt {|\lambda|}}\left(e^{\frac{1}{\sqrt{|\lambda|}}}-1-\frac{1}{\sqrt{|\lambda|}}\right).
\end{align*}
Hence,
\begin{equation*}
\|(|\lambda|-T)^{-1}\|\geq
\frac{1}{c}\frac{\sqrt{|\lambda|}}{1+\sqrt{|\lambda|}}\left(e^{\frac{1}{\sqrt{|\lambda|}}}-1-\frac{1}{\sqrt{|\lambda|}}\right),\ \ \lambda \neq 0,
\end{equation*}
where $c=\|\varphi_0^{(0)}\|$.
\end{proof}

\begin{prop}\label{Kej}
Let $T$ be as in Lemma \ref{T:quani},
then $k_{T}(e_j^{(m)})=\sqrt{r_j}$ for each $m\geq 0$.
\end{prop}

\begin{proof}
Since $e_j^{(m)}=u_{m,0}e_{j}^{(0)}$ for $m\geq 0$, we have $\|e_j^{(m)}\|\leq M \|e_{j}^{(0)}\|$.
Without loss of generality, one can assume that $\|e_{j}^{(0)}\|=1$ for each $j\geq 0$.
Thus, $\|e_j^{(m)}\|\leq M$ for $m,j\geq 0$.
Since
\begin{align*}
\|(\lambda-T)^{-1}e_j^{(m)}\|
&=\|\sum_{k=0}^{\infty}\frac{(2m)!{r_j}^k e_j^{(k+m)}}{(2k+2m)!\lambda^{k+1}}\|\\
&\leq
M\sum_{k=0}^{\infty}\frac{(2m)!{r_j}^k}{(2k+2m)!|\lambda|^{k+1}}\\
&\leq
M(2m)!\sum_{k=0}^{\infty}\frac{{r_j}^k}{(2k)!|\lambda|^{k+1}}\\
&=
M(2m)!\left(\sum_{k=1}^{\infty}\frac{\sqrt{{r_j}}^{2k}}{(2k)!\sqrt{|\lambda|}^{2k}}+\frac{1}{|\lambda|}\right)\\
&\leq
M(2m)!\left(e^{\frac{\sqrt{r_j}}{\sqrt{|\lambda|}}}+\frac{1}{|\lambda|}\right),
\end{align*}
it follows that
\begin{equation}\label{T:lnormalej}
\|(\lambda-T)^{-1}e_j^{(m)}\|\leq
M(2m)!\left(e^{\frac{\sqrt{r_j}}{\sqrt{|\lambda|}}}+\frac{1}{|\lambda|}\right).
\end{equation}
Combining \eqref{S4:lnorm} and \eqref{T:ej}, we have
\begin{align*}
k_{T}(e_j^{(m)})&=
\lim_{\lambda\rightarrow 0}
\frac{\ln \|(\lambda-T)^{-1}e_j^{(m)}\|}{\ln\|(\lambda-T)^{-1}\|}
\geq
\lim_{|\lambda|\rightarrow 0}
\frac{\ln \|(|\lambda|-T)^{-1}e_j^{(m)}\|}{\ln\|(|\lambda|-T)^{-1}\|}\\
&\geq\lim_{|\lambda|\rightarrow 0}
\frac{\ln\left(\frac{1}{c}\frac{(2m)!\sqrt{|\lambda|}|\lambda|^{m}}{\sqrt{r_j}(r_j^{m+1}+r_j^{m+1}\sqrt{|\lambda|})}
\left(e^{\frac{\sqrt{r_j}}{\sqrt{|\lambda|}}}-\sum_{i=0}^{2m+1}\frac{\sqrt{r_j}^i}{i!\sqrt{|\lambda|}^{i}}\right)\right)}
{\ln\left(\frac{1}{|\lambda|}+ \frac{M^2}{|\lambda|}e^{\frac{1}{\sqrt{|\lambda|}}}+
\frac{M^2}{(\sqrt{|\lambda|}+1)|\lambda|}e^{\frac{1}{\sqrt{|\lambda|}}}\right)}\\
&=\sqrt{r_j}.
\end{align*}
Combining \eqref{T:leqnorm} and \eqref{T:lnormalej}, we have
\begin{align*}
k_{T}(e_j^{(m)})=
\lim_{\lambda\rightarrow 0}
\frac{\ln \|(\lambda-T)^{-1}e_j^{(m)}\|}{\ln\|(\lambda-T)^{-1}\|}
&\leq
\frac{M(2m)!\left(e^{\frac{\sqrt{r_j}}{\sqrt{|\lambda|}}}+\frac{1}{|\lambda|}\right)}
{\ln\left(\frac{1}{c}\frac{\sqrt{|\lambda|}}{1+\sqrt{|\lambda|}}\left(e^{\frac{1}{\sqrt{|\lambda|}}}-1-\frac{1}{\sqrt{|\lambda|}}\right)\right)}\\
&=\sqrt{r_j}.
\end{align*}
Hence, $k_{T}(e_j^{(m)})=\sqrt{r_j}$ for each $m\geq 0$.
The proof is complete.
\end{proof}

\begin{prop}\label{R:lsubset}
Let $T$ be as in Lemma \ref{T:quani},
then $\sigma\subset \Lambda(T)$.
\end{prop}

\begin{proof}
By Proposition \ref{Kej}, we have $\{\sqrt {r_k}:k\geq 0\}\subset \Lambda(T)$.
Since $\Lambda(T)$ is right closed and the right closure of $\{\sqrt{r_k}: k\geq 0\}$ is $\sigma$, it holds that $\sigma\subset \Lambda(T)$.
\end{proof}

\begin{prop}\label{T:subsetl}
Let $T$ be as in Lemma \ref{T:quani}.
For any $x\in X$, let
$$l=\sup\{r_j: \text{there exists $m>0$ such that} ~~ f_{j}^{(m)}(x)\neq 0\},$$
then $k_{T}(x)\leq\sqrt{l}$.
\end{prop}
\begin{proof}
Let $x=\sum_{m=0}^{\infty}\sum_{j=0}^{\infty}f_{j}^{(m)}(x)e_j^{(m)}\in X$.
For each $k\geq 1$, we have
\begin{align*}
T^kx&=\sum_{n=0}^{\infty}\frac{(2n)!u_{n+k,0}A_0^ku_{0,n}}{(2n+2k)!}
\sum\limits_{m=0}^{\infty}\sum\limits_{j=0}^{\infty}f_{j}^{(m)}(x)e_j^{(m)}\\
&=\sum\limits_{m=0}^{\infty}\sum\limits_{j=0}^{\infty}\frac{(2m)!r_j^kf_{j}^{(m)}(x)e_j^{(k+m)}}{(2m+2k)!}.
\end{align*}
It follows that
\begin{align*}
\|(\lambda-T)^{-1}x\|=\|\sum_{k=0}^{\infty}\frac{T^k}{\lambda^{k+1}}x\|
&=\|\sum_{k=0}^{\infty}\sum\limits_{m=0}^{\infty}\sum\limits_{j=0}^{\infty}\frac{(2m)!r_j^kf_{j}^{(m)}(x)e_j^{(k+m)}}{(2m+2k)!\lambda^{k+1}}\|\\
&\leq
\sum_{k=0}^{\infty}\frac{l^k}{(2k)!|\lambda|^{k+1}}
\|\sum\limits_{m=0}^{\infty}\sum\limits_{j=0}^{\infty}f_{j}^{(m)}(x)e_j^{(k+m)}\|\\
&=
\sum_{k=0}^{\infty}\frac{l^k}{(2k)!|\lambda|^{k+1}}\|x\|\\
&=
\|x\|\sum_{k=0}^{\infty}\frac{\sqrt{l}^{2k}}{(2k)!\sqrt{|\lambda|}^{2k+2}}\\
&\leq \|x\|\left(e^{\frac{\sqrt{l}}{\sqrt{|\lambda|}}}+\frac{1}{|\lambda|}\right).
\end{align*}
Then, by Lemma \ref{T:normal}, we have
\begin{align*}
k_{T}(x)&=\limsup_{\lambda\rightarrow 0}
\frac{\ln\|(\lambda-T)^{-1}x\|}{\ln\|(\lambda-T)^{-1}\|}\\
&\leq
\limsup_{\lambda\rightarrow 0}
\frac{\ln\left(\|x\|\left(e^{\frac{\sqrt{l}}{\sqrt{|\lambda|}}}+\frac{1}{|\lambda|}\right)\right)}
{\ln\left(\frac{1}{c}\frac{\sqrt{|\lambda|}}{1+\sqrt{|\lambda|}}\left(e^{\frac{1}{\sqrt{|\lambda|}}}-1-\frac{1}{\sqrt{|\lambda|}}\right)\right)}=\sqrt{l}.
\end{align*}
This completes the proof that $k_T(x) \leq \sqrt{l}$.
\end{proof}

\begin{prop}\label{T:subsetr}
Let $T$ be as in Lemma \ref{T:quani}.
If there exists a non-negative integers $n_0$ such that $f_{j}^{(n_0)}(x)\neq 0$ for any $x\in X$,
then $k_{T}(x)\geq \sqrt{r_j}$.
\end{prop}

\begin{proof}
Since there exists a non-negative integers $n_0$ such that
$f_{j}^{(n_0)}(x)\neq 0$ for any $x\in X$,
one can assume that
$x=\sum_{m=n_0}^{\infty}\sum_{i=0}^{\infty}f_{i}^{(m)}(x)e_i^{(m)}$.
For $\lambda\neq 0$, we have
\begin{align*}
\varphi_{j}^{(n_0)}\left((\lambda-T)^{-1}x\right)
&=\varphi_{j}^{(n_0)}\left(\sum_{k=0}^{\infty}\frac{T^k}{\lambda^{k+1}}x\right)\\
&=\varphi_{j}^{(n_0)}\left(\sum_{k=0}^{\infty}\sum_{m=n_0}^{\infty}\sum_{i=0}^{\infty}\frac{(2m)!f_i^{(m)}(x)r_i^k{e_i}^{(k+m)}}{(2m+2k)!\lambda^{k+1}}\right)\\
&=\sum_{k=0}^{\infty}\sum_{m=n_0}^{\infty}\varphi_{j}^{(n_0)}\left(\sum_{i=0}^{\infty}\frac{(2m)!f_i^{(m)}(x)r_i^k{e_i}^{(k+m)}}{(2m+2k)!\lambda^{k+1}}\right)\\
&=\sum_{k=0}^{\infty}\sum_{m=n_0}^{\infty}\frac{(2m)!f_j^{(m)}(x)r_j^k\varphi_{j}^{(n_0)}\left({e_j}^{(k+m)}\right)}{(2m+2k)!\lambda^{k+1}}\\
&=\varphi_{j}^{(n_0)}\left(\sum_{k=0}^{\infty}\sum_{m=n_0}^{\infty}\frac{(2m)!f_j^{(m)}(x)r_j^k{e_j}^{(k+m)}}{(2m+2k)!\lambda^{k+1}}\right)\\
&=\varphi_{j}^{(n_0)}\left((\lambda-T)^{-1}\sum_{m=n_0}^{\infty}f_j^{(m)}(x){e_j}^{(m)}\right)\\
&=\sum_{m=n_0}^{\infty}f_j^{(m)}(x)\varphi_{j}^{(n_0)}\left((\lambda-T)^{-1}{e_j}^{(m)}\right)\\
&=\sum_{m=n_0}^{\infty}f_j^{(m)}(x)
\sum_{k=0}^{\infty}\frac{r_j^k (2m)!}{(2m+2k+2)!\lambda^{k+1}}\\
&=\sum_{m=n_0}^{\infty}(2m)!f_j^{(m)}(x){r_j}^{-m}\lambda^{m-1}\sum_{k=m}^{\infty}\frac{r_j^k}{(2k+2)!\lambda^{k}}\\
&=\sum_{k=n_0}^{\infty}\frac{{r_j}^k}{(2k+2)!\lambda^{k}}\left(\sum_{m=n_0}^{k}(2m)!f_j^{(m)}(x){r_j}^{-m}\lambda^{m-1}\right).
\end{align*}
We next compute the integral mean value of $\left|\varphi_{j}^{(n_0)}\left((\lambda-T)^{-1}x\right)\right|$.
Set $\lambda=te^{i\theta}$ with $t=|\lambda|$.
Then
\begin{align*}
&\frac{1}{2\pi}\int_{0}^{2\pi}\left|\varphi_{j}^{(n_0)}\left((\lambda-T)^{-1}x\right)\right| d\theta\\
&=
\frac{1}{2\pi}\int_{0}^{2\pi} \left|\sum\limits_{k=n_0}^{\infty}\frac{{r_j}^k}{(2k+2)!\lambda^{k}}\left(\sum_{m=n_0}^{k}(2m)!f_j^{(m)}(x){r_j}^{-m}\lambda^{m-1}\right)\right|d\theta\\
&\geq
\left|\sum\limits_{k=n_0}^{\infty}\frac{{r_j}^k}{(2k+2)!t^{k}}\frac{1}{2\pi}\int_{0}^{2\pi}\sum\limits_{m=n_0}^{k}(2m)!f_j^{(m)}(x){r_j}^{-m}t^{m-1}e^{i(m-n_0)\theta}\right|d\theta\\
&=
\left|\sum\limits_{k=n_0}^{\infty}\frac{{r_j}^k}{(2k+2)!t^{k}}\frac{1}{2\pi}\int_{0}^{2\pi}\sum\limits_{m=n_0}^{k}(2m)!f_j^{(m)}(x){r_j}^{-m}t^{m-1}e^{i(m-n_0)\theta}d\theta\right|\\
&=
\left|\sum\limits_{k=n_0}^{\infty}\frac{{r_j}^k}{(2k+2)!t^{k}}(2m_0)!f_j^{(n_0)}(x){r_j}^{-n_0}t^{n_0-1}\right|\\
&=\left|f_j^{(n_0)}(x)\sum\limits_{k=0}^{\infty}\frac{{r_j}^k(2n_0)!}{(2k+2n_0+2)!t^{k+1}}\right|\\
&=
\left|f_j^{(n_0)}(x)\right|\left|\varphi_j^{(n_0)}\left((t-T)^{-1}e_j^{(n_0)}\right)\right|.
\end{align*}
Thus, for $t>0$, we have
\begin{align*}
\sup_{\theta\in[0,2\pi]}\left|\varphi_{j}^{(n_0)}\left((te^{i\theta}-T)^{-1}x\right)\right|
&\geq{\frac{1}{2\pi}\int_{0}^{2\pi}\left|\varphi_{j}^{(n_0)}\left((te^{i\theta}-T)^{-1}x\right)\right| d\theta}\\
&\geq \left|f_j^{(n_0)}(x)\right|\left|\varphi_j^{(n_0)}\left((t-T)^{-1}e_j^{(n_0)}\right)\right|.
\end{align*}
Pick $\lambda=te^{i\theta}$ so that
$$\left|\varphi_{j}^{(n_0)}\left((\lambda-T)^{-1}x\right)\right|=\sup_{\theta\in[0,2\pi]}\left|\varphi_{j}^{(n_0)}\left((te^{i\theta}-T)^{-1}x\right)\right|.$$
Thus,
$$\left|\varphi_{j}^{(n_0)}\left((te^{i\theta}-T)^{-1}x\right)\right|\geq \left|f_j^{(n_0)}(x)\right|\left|\varphi_j^{(n_0)}\left((t-T)^{-1}e_j^{(n_0)}\right)\right|.$$
Notice that
\begin{align*}
\|(te^{i\theta}-T)^{-1}x\|
&\geq \frac{1}{c}\left|\varphi_{j}^{(n_0)}\left((te^{i\theta}-T)^{-1}x\right)\right|\\
&\geq \frac{1}{c}\left|f_j^{(n_0)}(x)\right|\left|\varphi_j^{(n_0)}\left((t-T)^{-1}e_j^{(n_0)}\right)\right|.
\end{align*}
Then, by Lemma \ref{T:normvector}, we have
\begin{align*}\label{F:re}
\|(te^{i\theta}-T)^{-1}x\|
\geq
\frac{1}{c}\left|f_j^{(n_0)}(x)\right|\frac{(2n_0)!\sqrt{|t|}|t|^{n_0}}{\sqrt{r_j}(r_j^{n_0+1}+r_j^{n_0+1}\sqrt{|t|})}
\left(e^{\frac{\sqrt{r_j}}{\sqrt{|t|}}}-\sum_{i=0}^{2n_0+1}\frac{\sqrt{r_j}^i}{i!\sqrt{|t|}^{i}}\right).
\end{align*}
It follows that
\begin{align*}
k_{T}(x)
&=\limsup_{\lambda\rightarrow 0}
\frac{\ln\|(\lambda-T)^{-1}x\|}{\ln\|(\lambda-T)^{-1}\|}\\
&\geq\limsup_{|\lambda|\rightarrow 0}
\frac{\ln\|(|\lambda|-T)^{-1}x\|}{\ln\|(|\lambda|-T)^{-1}\|}\\
&\geq\limsup_{t\rightarrow 0}
\frac{\ln\left(\frac{1}{c}\left|f_{j}^{n_0}(x)\right|\frac{(2n_0)!\sqrt{|t|}|t|^{n_0}}{\sqrt{r_j}\left(r_j^{n_0+1}+r_j^{n_0+1}\sqrt{|t|}\right)}
\left(e^{\frac{\sqrt{r_j}}{\sqrt{|t|}}}-\sum_{i=0}^{2n_0+1}\frac{\sqrt{r_j}^i}{i!\sqrt{|t|}^{i}}\right)\right)}
{\ln\left(\frac{1}{|t|}+ \frac{M^2}{|t|}e^{\frac{1}{\sqrt{|t|}}}+
\frac{M^2}{(\sqrt{|t|}+1)|t|}e^{\frac{1}{\sqrt{|t|}}}\right)}\\
&=\sqrt{r_j}.
\end{align*}
So, $k_{T}(x)\geq \sqrt{r_j}$.
The proof is complete.
\end{proof}

\begin{prop}\label{Prop2}
Let $T$ be as in Lemma \ref{T:quani}, then
$\Lambda(T)\subset \sigma$.
\end{prop}

\begin{proof}
By Proposition \ref{T:subsetl} and \ref{T:subsetr}, we have $k_{T}(x)=\sqrt{l}$ for any non-zero vector $x\in X$.
Since $\sigma$ is right closed, we have $\Lambda(T)\subset \sigma$.
\end{proof}

Now we are ready to give the proof of Theorem \ref{dlB}.

\begin{proof}[Proof of Theorem~\ref{dlB}]
Let $T$ be as in Lemma \ref{T:quani}.
By Proposition \ref{R:lsubset} and Proposition \ref{Prop2}, we have $\Lambda(T)=\sigma$.
\end{proof}



\end{document}